\documentclass[review,hidelinks,onefignum,onetabnum]{siamart250211}

\nolinenumbers

\usepackage{lipsum}
\usepackage{amsfonts}
\usepackage{graphicx}
\usepackage{epstopdf}
\usepackage{comment}

\usepackage{amssymb,amsfonts,amsmath,latexsym,dsfont}
\usepackage{mathtools}
\usepackage{mathrsfs}
\usepackage[algo2e,ruled,vlined]{algorithm2e}
\usepackage{setspace}
\usepackage{subfigure}
\usepackage{fancyhdr}
\usepackage{wrapfig}

\usepackage{booktabs}

\usepackage[normalem]{ulem}

\usepackage{bm}
\renewcommand{\vec}[1]{{\bm #1}}
\usepackage{boxedminipage}
\usepackage{pgf,pgfarrows} 
\usepackage{subfigure} 

\newcommand{\FrameboxA}[2][]{#2}
\newcommand{\Framebox}[1][]{\FrameboxA}

\DeclareMathOperator*{\argmin}{arg\,min}

\newcommand{\norme}[1]{\left\lVert#1\right\rVert}

\newcommand{\ii}{\mathrm{i}}

\newcommand{\grad}{\nabla}

\newcommand{\bfu}{{\bf u}}

\renewcommand{\theequation}{\arabic{section}.\arabic{equation}}

\newcommand{\R}{\ensuremath{\mathds{R}}}

\definecolor{darkblue}{rgb}{0.08, 0.15, 0.48}

\newtheorem{rem}[theorem]{Remark}
\newtheorem{thm}[theorem]{Theorem}

\graphicspath{{images/}{./}}

\ifpdf
  \DeclareGraphicsExtensions{.eps,.pdf,.png,.jpg}
\else
  \DeclareGraphicsExtensions{.eps}
\fi

\headers{Dispersion Correction for Elastic Helmholtz}{P.H. Cocquet, A. Tonnoir and R. Yovel}

\title{Asymptotic dispersion correction for the isotropic elastic Helmholtz equation discretized with a MAC scheme
\thanks{Corresponding author: . \funding{RY is supported by the Israel Science Foundation (grant No. 656/23), the Ariane de Rothschild scholarship and by Kreitman High-tech scholarship. The authors also thank the Lynn and William Frankel Center for Computer Science at BGU.}}}

\author{Pierre-Henri Cocquet\thanks{Laboratoire des Sciences pour l'Ing{\'e}nieur Appliqu{\'e}es {\`a} la M{\'e}canique et au g{\'e}nie {\'E}lectrique (SIAME), Universit{\'e} de Pau et des Pays de l'Adour, Pau, France. \email{pierre-henri.cocquet@univ-pau.fr}} \and Antoine Tonnoir\thanks{Normandie University, INSA de Rouen Normandie, LMI, Rouen, France. \email{antoine.tonnoir@insa-rouen.fr}}
\and Rachel Yovel\thanks{Institute for Interdisciplinary Computational Science, the Stein Faculty of Computer and Information Science, Ben-Gurion University of the Negev, Beer-Sheva, Israel.
  \email{yovelr@bgu.ac.il}}}

\usepackage{amsopn}

\begin{document}

\maketitle

\begin{abstract}
The numerical simulation of time-harmonic wave propagation in elastic media plays an important role in applications such as geophysics and non-destructive testing.
Accurate discretization of the elastic Helmholtz equation at high frequencies is challenging due to numerical dispersion and pollution effects.
In this work, we develop an asymptotic dispersion correction for a Marker-And-Cell (MAC) discretization of the isotropic elastic Helmholtz equation.
We characterize the discrete dispersion relation of the scheme and determine the leading-order term in the dispersion error in both two and three spatial dimensions.
Based on this analysis, we derive a correction that is asymptotically optimal in the limit of vanishing mesh size.
The proposed approach improves the agreement between the discrete and continuous wave propagation properties while preserving the structure of the underlying discretization.
We also establish a connection between the factorization of the dispersion relation and the structure of the grad-div operator symbol, providing additional insight into the algebraic structure of the elastic problem. Numerical experiments finally demonstrate a substantial reduction of relative errors and confirm the effectiveness of the proposed correction.
We further provide numerical evidence that the corrected discretization improves the convergence behavior of multigrid solvers.
\end{abstract}



\section{Introduction}\label{sec:intro}

The general motivation of this work is the development of efficient numerical tools for simulating time-harmonic wave propagation in elastic media. 
More precisely, we are interested in time-harmonic waves propagation in (isotropic) elastic media and emphasize that the frequency-domain formulation is particularly attractive in applications involving imaging and inversion, where the response of the medium must be computed repeatedly for many frequencies.
This configuration indeed occurs in several contexts such as geophysics \cite{virieux2011review,operto2009finite} or Non Destructive Testing \cite{baronian2018linear,blitz1995ultrasonic}.

Solving efficiently and accurately time-harmonic waves equation is known to be a difficult task and many works have been dedicated to the Helmholtz equation which serves as a standard model. In particular, a central difficulty is the pollution effect \cite{ihlenburg1995dispersion,babuska1997pollution,gander2026fourier} which implies that even for a fixed number of points per wavelength, the error grows as the frequency increases. Therefore, handling the pollution effect requires considering a large number of degree of freedom so that the resulting linear systems become increasingly difficult to solve as the frequency grows. Several remedies have been studied for the Helmholtz equation, for instance high order method \cite{spence2023simple} or Boundary Element Method \cite{galkowski2023does}.

Among the various approaches proposed to mitigate pollution, dispersion correction methods have attracted considerable attention because they directly target the mismatch between the continuous and discrete propagation of waves. 
By modifying the discretization so as to improve its dispersion relation, these methods can significantly reduce phase errors while preserving the overall structure of the numerical scheme. 
The vast majority of works deal with the Helmholtz equation and the most common method (see \cite{stolk2014multigrid,wu2018optimal,liu2015optimized,jo1996optimal,wu2021new,cheng2017dispersion,chen2012dispersion,wu2014dispersion}) is to consider finite differences (FD) stencil with free parameters that are next determined numerically by minimizing the dispersion error. More recently, some asymptotically optimal dispersion corrections have been proposed for the Helmholtz equation in \cite{cocquet2021closed,cocquet2024asymptotic} and for the Maxwell equation in \cite{cocquet2025improving}. This method introduces, in a given FD scheme, some free parameters through consistent perturbation of some physical parameter (e.g. the wavenumber, the angular frequency). Assuming a large enough number of grid point per wavelength is considered, the perturbation that minimizes the dispersion error is then determined in close-form. The advantages of this dispersion minimizing technique are twofold: first, it can be applied to FD scheme that do not have any free parameters, and second, it does not rely on any numerical optimization.

In this work, we are interested in dispersion correction for the vectorial elastodynamics equation (a.k.a the elastic Helmholtz equation).
While some works suggests dispersion-minimizing discretizations for elastic wave propagation problems, see \cite{levander1988fourth, gosselin20143d,harari2011stabilized},
to the best of our knowledge there exists no works suggesting an asymptotic dispersion correction via modified physical parameters.
Our purpose is to extend the method of \emph{asymptotic dispersion correction} (see \cite{cocquet2024asymptotic,cocquet2025improving})
in this setting which is more challenging due to the coexistence of pressure and shear waves with different propagation velocities.
More precisely, 
we propose an asymptotic dispersion correction for the isotropic elastic Helmholtz equation discretized with a Marker And Cell (MAC) scheme. 
We characterize the discrete dispersion relation as well as the order of the dispersion error. 
We derive a dispersion correction which is optimal in the asymptotic case where $h\to 0$, in two dimensions and three dimensions. Our theoretical results are supported by numerical evidence demonstrating the effectiveness of the dispersion correction. We also give numerical evidence for the improvement of multigrid convergence when using our dispersion correction. 
Finally, we establish a connection between the factorization of the discrete dispersion correction and the structure of the grad-div symbol.

The organization of the paper is as follows: 
in Section \ref{sec:cont_eq} we characterize the symbol of the continuous elastic Helmholtz equation in isotropic media. 
In Section \ref{sec:mac} we introduce the MAC finite difference scheme and analyze its dispersion properties. 
Section \ref{sec:asymptotic} is devoted to the derivation of the asymptotic dispersion correction for the MAC scheme while Section \ref{sec:extension} provides some extensions and limits of the proposed approach. In Section \ref{sec:Numerics}, we give numerical experiments to show the effect of the dispersion correction on the relative error and on the convergence behavior of a multigrid algorithm. This paper then ends with some conclusions and future works.

\section{Dispersion analysis at the continuous level}
\label{sec:cont_eq}

Let $\Omega$ be a bounded open set of $\R^d$.
We consider the time-harmonic isotropic elastic wave equation:
(see \cite[Eq. (2.2)]{yovel2024lfa})
\begin{equation}\label{eq:Elastic_Helmholtz_inhomogeneous_Lame}
-\textbf{div}\left\{\mu \left(\nabla \vec{u}+ \nabla\vec{u}^T \right)\right\} - \nabla \lambda \mathrm{div}\,\vec{u} 
- \rho(x)\omega^2 \vec{u} = \vec{f}\ \text{in}\ \Omega,
\end{equation}
where $\vec{u}$ is the displacement field written as $ \vec{u}= (u^x, u^y)$ for $d=2$ and $\vec{u}= (u^x, u^y,u^z)$ for $d=3$, $\lambda,\mu$ are the Lam\'e coefficients and $\vec{f}$ is a given source term. 
Another common formulation of the isotropic elastic Helmholtz equation is
\begin{equation}\label{eq:elastic_Helm_gradiv_formulation}
    -\textbf{div}\left(\mu \grad\right)\vec{u} - \grad(\lambda+\mu)\text{div} \vec{u} - \rho(x)\omega^2 \vec{u}
 = \vec{f}
 \end{equation}
 which is equivalent to \eqref{eq:Elastic_Helmholtz_inhomogeneous_Lame} in the case of constant $\mu$. We will consider for the dispersion analysis an homogeneous medium so that $\lambda$, $\mu$ and $\rho$ are supposed to be constant positive real numbers and the source term $\vec{f} = 0$. In that case, using the identity 
$$
\textbf{div}\left(\nabla \vec{u}+ \nabla\vec{u}^T \right) = \vec{\Delta}\vec{u} + \nabla \mathrm{div}\,\vec{u},
$$
we can rewrite \eqref{eq:Elastic_Helmholtz_inhomogeneous_Lame} or, equivelntly, \eqref{eq:elastic_Helm_gradiv_formulation} as (see \cite[Eq. (2.3)]{yovel2024lfa})
\begin{equation}\label{eq:Elastic_Helmholtz_homogeneous_mu}
-\mu\vec{\Delta}\, \vec{u} - \nabla \left(\lambda +\mu\right)\mathrm{div}\,\vec{u} 
- \rho \omega^2 \vec{u} = \vec{0},
\end{equation}
where $\vec{\Delta}$ is the vector Laplacian which is simply defined as the scalar Laplacian applied component-wise, hence $\vec{\Delta}\vec{u} = \left(\Delta u^x,\Delta u^y \right)$ in two dimensions and $\vec{\Delta}\vec{u} = \left(\Delta u^x,\Delta u^y,\Delta u^z \right)$ in three dimensions. 


\begin{rem} 
The Poisson ratio, representing the stress-strain relationship in the media, is given by 
$$
\nu = \frac{\lambda}{2(\lambda+\mu)},
$$
and, when $\nu\to \frac{1}{2}$ the material is nearly incompressible.
\end{rem}


To perform the dispersion analysis, we look for a plane-wave solution to \eqref{eq:Elastic_Helmholtz_homogeneous_mu} hence we want to find a non-trivial solution of the form 
$$
\vec{u} = \vec{X} \mathrm{e}^{\ii \vec{\xi}\cdot \vec{x}}
$$
where $\vec{X}\in\mathbb{C}^d$ is the polarization vector. Inserting the plane-wave into \eqref{eq:Elastic_Helmholtz_homogeneous_mu} yields 
$$
\sigma\left( \vec{\xi}\right)\vec{X} = \vec{0},
$$
with the matrix $\sigma(\vec{\xi})$
\begin{equation}\label{eq:symb-cont}
\sigma\left(\vec{\xi}\right) =  \mu \norme{\vec{\xi}}^2 I_d + (\lambda + \mu) \vec{\xi}\vec{\xi}^T - \rho\omega^2 I_d
\end{equation}
being the symbol associated to the elastic Helmholtz operator (which can be obtained directly by taking the Fourier transform of \eqref{eq:Elastic_Helmholtz_homogeneous_mu}). The continuous dispersion relation is defined by the set of value $\vec{\xi}$ for which we have non trivial plane-wave solutions, that is to say: 
$$
D_c := \left\{\vec{\xi} \in \mathbb{R}^d \ |\ \mathrm{det}\, \sigma\left( \vec{\xi}\right) = 0 \right\}, \quad d\in \{2,3\}. 
$$ 
We recall the following formula 
\begin{equation}\label{eq:Formula_det}
\mathrm{det}\left(A + \vec{c}\vec{d}^T \right) = \mathrm{det}(A)\left(1+\vec{d}^T A^{-1}\vec{c} \right),
\end{equation}
which holds for any invertible matrix $A\in \mathbb{C}^{N\times N}$ and any vectors $\vec{c},\ \vec{d}\in \mathbb{C}^{N\times 1}$. 
Using \eqref{eq:Formula_det} with $A = \mu \norme{\vec{\xi}}^2 I_d - \rho\omega^2 I_d$ and $\vec{c}=\vec{d}=\sqrt{\lambda+\mu}\,\vec{\xi}$  then gives: 
$$
\mathrm{det}\, \sigma\left( \vec{\xi}\right) = \left( 
 \mu \norme{\vec{\xi}}^2 -  \rho\omega^2
\right)^{d-1}\left((\lambda+2\mu)\norme{\vec{\xi}}^2-\rho \omega^2 \right).
$$
As a consequence, the dispersion relation is the reunion of two circles (for $d=2$) and 2 spheres (for $d=3$) with radii
$$
R_S = \omega\sqrt{\frac{\rho}{\mu}},\quad  R_P =  \omega \sqrt{\frac{\rho}{\lambda+2\mu}}. 
$$
Let us also introduce the scalar wavenumber $k$ defined as 
$$
\forall \vec{\xi}\in D_c:\ k = \norme{\vec{\xi}}.
$$
For the elastic-Helmholtz equation, one then has two scalar wavenumbers given by  
$$
k_S =  \omega\sqrt{\frac{\rho}{\mu}}\ \text{and}\ k_P = \omega\sqrt{\frac{\rho}{\lambda+2\mu}},
$$
where $k_S$ is the wavenumber associated to shear waves (whose polarization vector is orthogonal to $\vec{\xi}$) and $k_P$ to pressure waves (whose polarization vector is colinear to $\vec{\xi}$).

\section{MAC finite-differences scheme and its properties}
\label{sec:mac}

This section is dedicated to present the MAC scheme, some of its properties and its dispersion analysis. 

\subsection*{Two-dimensional MAC scheme}
We follow the presentation of \cite{rui2018locking} to present the MAC scheme although we slightly modify their notations. We will first detail the $2$D case and consider the case of the whole plane $\mathbb{R}^2$ for the dispersion analysis. In $2$D, we have $\vec{u}=\left(u^x,u^y\right)$ so that the isotropic elastodynamic equations are:
\begin{equation}\label{eq:split}
    \left\{ \begin{array}{l}
         -\mu \Delta u^x - (\lambda + \mu)\partial_{xx} u^x - \rho \omega^2 u^x - (\lambda + \mu)\partial_x \partial_y u^y = f^x, \\[5pt]
         -\mu \Delta u^y - (\lambda + \mu)\partial_{yy} u^y - \rho \omega^2 u^y - (\lambda + \mu)\partial_y \partial_x u^x = f^y,
    \end{array} \right. \quad \text{in}\ \Omega.
\end{equation}
The MAC scheme considers two grids defined by
\begin{equation}
\begin{split}
    & \mathcal{G}_{1} = \left\{ (x_{i+\frac{1}{2}},y_j) = \left(\left(i+\frac{1}{2}\right)h_x,jh_y\right), \;\; (i,j) \in \mathbb{Z}^2\right\}, \\
    \text{and} \quad & \mathcal{G}_{2} = \left\{ (x_{i},y_{j+\frac{1}{2}}) = \left(ih_x,\left(j+\frac{1}{2}\right)h_y\right), \;\; (i,j) \in \mathbb{Z}^2\right\},
\end{split}
\end{equation}
so that $u^x$ is approximated on $\mathcal{G}_1$ and $u^y$ on $\mathcal{G}_2$. To facilitate the notations, we introduce the following backward and forward stencils
\begin{eqnarray*}
\delta_x^{+} v_{i,j} &=& \frac{1}{h_x}\left(v_{i+1,j}-v_{i,j}\right),\ \delta_x^{-} v_{i,j} =\frac{1}{h_x}\left(v_{i,j}-v_{i-1,j}\right),\ 
\\
\delta_y^{+} v_{i,j} &=& \frac{1}{h_y}\left(v_{i,j+1}-v_{i,j}\right),\ \delta_y^{-} v_{i,j} = \frac{1}{h_y}\left(v_{i,j}-v_{i,j-1}\right). 
\end{eqnarray*}
With these notations, the MAC scheme can be written in compact form as 
\begin{equation}\label{eq:MAC_FD_Grid_AT}
\begin{array}{rcl}
-\mu \Delta_h u^x_{i+\frac{1}{2},j} -(\lambda+\mu) \left(\delta_x^+\delta_x^- u^x_{i+\frac{1}{2},j} + \delta_x^{+}\delta_y^{-}u^y_{i,j+\frac{1}{2}}\right)-\rho\omega^2 u^x_{i+\frac{1}{2},j}
&=& f_{i+\frac{1}{2},j}^x,
\\[0.5em]
-\mu \Delta_h u^y_{i,j+\frac{1}{2}} 
-(\lambda+\mu) \left(\delta_y^+\delta_y^- u^y_{i,j+\frac{1}{2}} + \delta_x^{-}\delta_y^{+}u^x_{i+\frac{1}{2},j}\right)-\rho\omega^2 u^y_{i,j+\frac{1}{2}}
&=& f^y_{i,j+\frac{1}{2}},
\end{array}
\end{equation}
where the discrete Laplacian is $
\Delta_h =\left(\delta_x^+\delta_x^- +\delta_y^-\delta_y^- \right)$.
\begin{rem}
We could have also considered approximating $u^x$ on $\mathcal{G}_2$ and $u^y$ on $\mathcal{G}_1$. In that case, the MAC scheme would become
\begin{equation}\label{eq:MAC_FD_Grid_2}
\begin{array}{rcl}
-\mu \Delta_h u^x_{i,j+\frac{1}{2}} -(\lambda+\mu) \left(\delta_x^+\delta_x^- u^x_{i,j+\frac{1}{2}} + \delta_x^{-}\delta_y^{+}u^y_{i+\frac{1}{2},j}\right)-\rho\omega^2 u^x_{i,j+\frac{1}{2}}
&=& f_{i,j+\frac{1}{2}}^x,
\\[0.5em]
-\mu \Delta_h u^y_{i+\frac{1}{2},j} 
\nonumber -(\lambda+\mu) \left(\delta_y^+\delta_y^- u^y_{i+\frac{1}{2},j} + \delta_x^{+}\delta_y^{-}u^x_{i,j+\frac{1}{2}}\right)-\rho\omega^2 u^y_{i+\frac{1}{2},j}
&=& f^y_{i+\frac{1}{2},j}.
\end{array}
\end{equation}
\end{rem}
To clarify the notations above, we give below the MAC scheme \eqref{eq:MAC_FD_Grid_AT} in expanded form:
\begin{equation}\label{eq:MAC_expanded_1st_eq}
\begin{split}
    & \mu \dfrac{u^x_{i+\frac{3}{2},j} + u^x_{i-\frac{1}{2},j} + u^x_{i+\frac{1}{2},j+1} + u^x_{i+\frac{1}{2},j-1} - 4 u^x_{i+\frac{1}{2},j}}{h_x^2}\\
    + \; &  (\lambda + \mu) \dfrac{u^x_{i+\frac{3}{2},j} + u^x_{i-\frac{1}{2},j} - 2u^x_{i+\frac{1}{2},j} }{h_x^2} + \rho \omega^2 u^x_{i+\frac{1}{2},j} \\
    + \; & (\lambda + \mu)  \dfrac{u^y_{i+1,j+\frac{1}{2}} + u^y_{i,j-\frac{1}{2}} - u^y_{i+1,j-1/2} - u^y_{i,j+\frac{1}{2}}}{h_x h_y} = f^x_{i+\frac{1}{2},j},
\end{split}    
\end{equation}
for the first equation in \eqref{eq:split}, and:
\begin{equation}\label{eq:MAC_expanded_2nd_eq}
\begin{split}
    & \mu \dfrac{u^y_{i+1,j+\frac{1}{2}} + u^y_{i-1,j+\frac{1}{2}} + u^y_{i,j+\frac{3}{2}} + u^y_{i,j-\frac{1}{2}} - 4 u^y_{i,j+\frac{1}{2}}}{h_y^2}\\
    + \; &  (\lambda + \mu) \dfrac{u^y_{i,j+\frac{3}{2}} + u^y_{i,j-\frac{1}{2}} - 2u^y_{i,j+\frac{1}{2}} }{h_y^2} + \rho \omega^2 u^y_{i,j+\frac{1}{2}} \\
    + \; & (\lambda + \mu) \dfrac{u^x_{i+\frac{1}{2},j+1} + u^x_{i-\frac{1}{2},j} - u^x_{i-\frac{1}{2},j+1} - u^x_{i+\frac{1}{2},j}}{h_xh_y} = f^y_{i,j+\frac{1}{2}},
\end{split}    
\end{equation}
for the second equation in \eqref{eq:split}.

Using Taylor expansions, one can show that the FD scheme \eqref{eq:MAC_expanded_1st_eq}, \eqref{eq:MAC_expanded_2nd_eq} has a second order truncation error.

\subsection*{Three-dimensional MAC scheme}\label{rem:3d_MAC_stencil}
Introducing the notations
$$
\nabla_h^- v_{i,j,k}= \left( 
\begin{array}{c}
\delta_x^-v_{i,j,k} \\ \delta_y^-v_{i,j,k} \\ \delta_z^-v_{i,j,k}
\end{array}
\right),\ \mathrm{div}_h^+ \vec{v_{i,j,k}} = \delta_x^+v^x_{i,j,k} + \delta_y^+v^y_{i,j,k}+\delta_z^+v^z_{i,j,k},
$$
the $3$D MAC scheme is simply given by
$$
\left(-\mu \vec{\Delta}_h -(\lambda+\mu)\nabla_h^-\mathrm{div}_h^+ -\rho\omega^2 \right)
\left( 
\begin{array}{c}
u^x_{i+1/2,j,k} \\ u^y_{i,j+1/2,k} \\ u^z_{i,j,k+1/2}
\end{array}
\right)
= 
\left(\begin{array}{c}
f^x_{i+1/2,j,k} \\ f^y_{i,j+1/2,k} \\ f^z_{i,j,k+1/2}
\end{array} \right)
$$
where $\Delta_h = \mathrm{div}_h^+\nabla_h^-= \mathrm{div}_h^-\nabla_h^+$ is the scalar Laplacian and $\vec{\Delta}_h$ is the vector Laplacian applied component-wise to any vector field. Using Taylor expansions show that it also has a second-order truncation error. 

\subsection*{Discrete Helmholtz decomposition}

In homogeneous media, it is classical (see e.g. \cite{bramble2008note}, \cite[Section 1.5.1]{martin2006multiple}) to decompose the displacement field $\vec{u}$ satisfying \eqref{eq:Elastic_Helmholtz_homogeneous_mu} as 
\begin{equation}\label{eq:Helmholtz_decomposition_continuous}
\vec{u} =  \vec{\varphi}  -\nabla \psi,
\end{equation}
where $\mathrm{div}\, \vec{\varphi}=0$ and $\nabla\times \nabla \psi=0$. In addition, $\psi$ is associated to the $P$-waves while $\vec{\varphi}$ is associated to the $S$-waves since 
$$
-\Delta \psi - k_P^2 \psi = 0,\qquad -\vec{\Delta}\vec{\varphi}-k_S^2\vec{\varphi} = 0.
$$
The decomposition \eqref{eq:Helmholtz_decomposition_continuous} is known as the (Hodge)-Helmholtz decomposition and we are going to show that it also holds for the displacement field discretized using the MAC scheme. 

We will show the result in the $3$D-case, for the $2$D case, see Remark \ref{rem:2D-Decomp-Helm-Discr}. 
Let $\vec{u}_h=\left(u^x_{i+\frac{1}{2},j,k} , u^y_{i,j+\frac{1}{2},k} , u^z_{i,j,k+\frac{1}{2}} \right)^T$ be the discrete displacement field that satisfies 
\begin{equation}\label{eq:Discrete_Elast_Helmholtz_3d}
\left(-\mu \vec{\Delta}_h -(\lambda+\mu)\nabla_h^-\mathrm{div}_h^+ -\rho\omega^2 \right)
\vec{u}_h
= 
\vec{f}_h. 
\end{equation}
From \cite[Theorem 3.1]{cocquet2025improving}, we have the following formulas (a discrete De Rham complex)
\begin{eqnarray}\label{eq:discrete_de_Rham}
&& \mathrm{div}_h^{\pm} \nabla^{\pm}_h\times =0,\ \nabla_h^\pm\times \nabla_h^\pm =0,\ \mathrm{div}_h^{\pm} \nabla^{\mp}_h =\Delta_h,
\\
\nonumber && \vec{\Delta}_h = \nabla_h^{+}\mathrm{div}^-_h - \nabla_h^-\times \nabla_h^+\times,\ \vec{\Delta}_h = \nabla_h^{-}\mathrm{div}^+_h - \nabla_h^+\times \nabla_h^-\times.
\end{eqnarray}
The finite differences operators above are defined at each grid point as 
\begin{eqnarray*}
&& \mathrm{div}^\pm_h \vec{\Phi} = \delta_x^\pm \Phi^x_{i,j,k} + \delta_y^\pm\Phi^y_{i,j,k}+\delta_z^\pm\Phi^z_{i,j,k},
\\
&& \nabla_h^\pm\times \vec{\Phi} = [\vec{e}_x\times ]\delta_x^\pm \vec{\Phi}_{i,j,k} + [\vec{e}_y\times ]\delta_y^\pm \vec{\Phi}_{i,j,k} +  [\vec{e}_z\times ]\delta_z^\pm \vec{\Phi}_{i,j,k},
\end{eqnarray*}
where $[\vec{X}\times]\in \R^{3\times 3}$ is defined through the cross product as $[\vec{X}\times] \vec{Y}=\vec{X}\times \vec{Y}$ for all $\vec{X},\vec{Y}\in\R^3$ and $\vec{e}_x,\vec{e}_y,\vec{e}_z$ are the vectors of the canonical basis of $\R^3$. We follow \cite{bramble2008note} and consider 
$$
\psi_h=\mathrm{div}_h^+\,\vec{u}_h.
$$
Using \eqref{eq:discrete_de_Rham} and \eqref{eq:Discrete_Elast_Helmholtz_3d}, we get that 
$$
-\mu \nabla_h^-\psi_h+\mu \nabla_h^+\times\nabla_h^-\times \vec{u}_h -(\lambda+\mu)\nabla_h^-\psi_h - \rho\omega^2\vec{u}_h = \vec{f}_h.
$$
Taking the discrete divergence $\mathrm{div}_h^+$ of the previous equation and using once again \eqref{eq:discrete_de_Rham} yield
\begin{equation}\label{eq:Discrete_P_waves}
 - (\lambda+2\mu) \Delta_h \psi_h - \rho\omega^2 \psi_h = \mathrm{div}_h^+\, \vec{f}_h.
\end{equation}
Note that \eqref{eq:Discrete_P_waves} is a discrete Helmholtz equation with wavenumber $k_P$ hence being associated to the $P$-waves. Let $\vec{\Phi}_h$ be defined as 
$$
\vec{\Phi}_h = \vec{u}_h + \alpha \nabla_h^-\psi_h + \beta \vec{f}_h,
$$
for some constants $\alpha,\beta$. 
One has 
$$
\mathrm{div}_h^+\, \vec{\Phi}_h = \psi_h + \alpha \Delta_h\psi_h + \beta \mathrm{div}_h^+\,\vec{f}_h =0,
$$
for 
$$
\alpha =\frac{1}{k_P^2} ,\ \beta = \frac{1}{\rho \omega^2}. 
$$
Taking the discrete gradient $\nabla_h^-$ of \eqref{eq:Discrete_P_waves} shows that each component of $\nabla_h^-\psi_h$ satisfies 
$$
- (\lambda+2\mu) \vec{\Delta}_h \left(\nabla_h^-\psi_h\right) - \rho\omega^2 \left(\nabla_h^-\psi_h\right) = \nabla_h^-\mathrm{div}_h^+\, \vec{f}_h.
$$
Using \eqref{eq:discrete_de_Rham} and more precisely that $\nabla_h^-\times\nabla_h^-=0$, one gets
$$
-\mu \vec{\Delta}_h \left(\nabla_h^-\psi_h\right) - (\lambda+\mu)\nabla_h^-\mathrm{div}_h^+ \left(\nabla_h^-\psi_h\right) - \rho\omega^2\left(\nabla_h^-\psi_h\right) = \nabla_h^-\mathrm{div}_h^+\, \vec{f}_h,
$$
hence $\nabla_h^-\psi_h$ satisfies \eqref{eq:Discrete_Elast_Helmholtz_3d} with a different right-hand-side. As a result, recalling that $\mathrm{div}_h^+\vec{\Phi}_h=0$ and applying the discrete elasticity operator, we obtain that $\vec{\Phi}_h$ satisfies the 
following equation 
$$
-\mu \vec{\Delta}_h \vec{\Phi}_h - \rho\omega^2 \vec{\Phi}_h = 
\underbrace{\vec{f}_h  +\alpha \nabla_h^-\mathrm{div}_h^+\, \vec{f}_h+\beta  \left(-\mu \vec{\Delta}_h -(\lambda+\mu)\nabla_h^-\mathrm{div}_h^+ -\rho\omega^2 \right)\vec{f}_h}_{:=\vec{g}_h}. 
$$
From the definition of $\alpha,\beta$, the previous right-hande side simplifies to 
$$
\vec{g}_h = -\frac{\mu}{\rho\omega^2} \vec{\Delta}_h\vec{f}_h + \frac{\mu}{\rho\omega^2} \nabla_h^-\mathrm{div}_h^+\vec{f}_h = -\frac{\mu}{\rho\omega^2} \nabla_h^+\times \nabla_h^-\times \vec{f}_h.
$$
Therefore, $\vec{\Phi}_h$ satisfies 
\begin{equation}\label{eq:Discrete_S_waves}
-\mu \vec{\Delta}_h \vec{\Phi}_h - \rho\omega^2 \vec{\Phi}_h =-\frac{\mu}{\rho\omega^2} \nabla_h^+\times \nabla_h^-\times \vec{f}_h,
\end{equation}
which is a vector Helmholtz equation associated to the wavenumber $k_S$ hence modelling the $S$-waves. Taking the discrete divergence $\mathrm{div}_h^{+}$ of \eqref{eq:Discrete_S_waves} and using \eqref{eq:discrete_de_Rham}, we retrieve that $\mathrm{div}_h^{+}\,\vec{\Phi}_h=0$. 

To summarize, we show in this section that any solution to the discrete elastic-Helmholtz equation \eqref{eq:Discrete_Elast_Helmholtz_3d} can be written as 
\begin{equation}\label{eq:discrete_Helmholtz_decomposition}
 \vec{u}_h =\vec{\Phi}_h - \alpha \nabla_h^-\psi_h - \beta \vec{f}_h,
\end{equation}
where $\psi_h$ satisfies a scalar Helmholtz equation with wavenumber $k_P$ (see \eqref{eq:Discrete_P_waves}) and $\vec{\Phi}_h$ is so that $\mathrm{div}_h^{+}\,\vec{\Phi}_h=0$ and satisfies a vector Helmholtz equation with wavenumber $k_S$ (see \eqref{eq:Discrete_S_waves}). In addition, for $\vec{f}_h = 0$, we retrieve the discrete analog the the classical Hodge-Helmholtz decomposition (see \eqref{eq:Helmholtz_decomposition_continuous}). 

\begin{rem}\label{rem:2D-Decomp-Helm-Discr}
All the previous developments can be also done in two dimensions mostly by setting the third component of $\vec{u}_h$ to $0$. In addition, the previouses computations can be done for any FD discretization for which we have discrete De Rham complex \eqref{eq:discrete_de_Rham}. 
\end{rem}

\begin{rem}
 Given the discrete Helmholtz decomposition \eqref{eq:discrete_Helmholtz_decomposition}, one may consider solving the elastic Helmholtz equation by solving two Helmholtz equations (one for $\vec{\Phi}_h$ and one for $\psi_h$), especially when $k_S \gg k_P$. 
 Indeed, this allows to avoid a large cost for computing $S$ waves \cite{burel2012solving,albella2018solving} by considering different discretization grids for the two equations. 
 Yet, let us emphasize that for a problem in a bounded domain $\Omega$, the boundary conditions couple these two PDEs. Hence, taking advantage of this decomposition is not trivial. This idea has not been investigated further and is beyond the scope of this work. 
\end{rem}

\subsection*{Dispersion analysis}

We are now in position to perform the discrete dispersion analysis. For simplicity, we will give the explanations in the $2$D case. The generalization to the $3$D case is straightforward and briefly explained in the Remark \ref{rem:3d_MAC} below.

The symbols associated to the stencils can be computed from the identities 
\begin{eqnarray*}
\delta_{x}^+ \phi_{\vec{\xi}}(x_i,y_j) &=& \underbrace{\left(\frac{\mathrm{e}^{\ii \xi_x h_x}-1}{h_x}\right)}_{=: \sigma_{\delta_x^+(\vec{\xi})}} \phi_{\vec{\xi}}(x_i,y_j),
\
\delta_{x}^- \phi_{\vec{\xi}}(x_i,y_j)= \underbrace{\left(\frac{1-\mathrm{e}^{-\ii \xi_x h_x}}{h_x} \right)}_{=: \sigma_{\delta_x^-(\vec{\xi})}} \phi_{\vec{\xi}}(x_i,y_j),
\\
\delta_{y}^+ \phi_{\vec{\xi}}(x_i,y_j) &=&\underbrace{\left( \frac{\mathrm{e}^{\ii \xi_y h_y}-1}{h_y} \right)}_{=: \sigma_{\delta_y^+(\vec{\xi})}}\phi_{\vec{\xi}}(x_i,y_j),
\
\delta_{y}^- \phi_{\vec{\xi}}(x_i,y_j) = \underbrace{\left(\frac{1-\mathrm{e}^{-\ii \xi_y h_y}}{h_y} \right)}_{=: \sigma_{\delta_y^-(\vec{\xi})}} \phi_{\vec{\xi}}(x_i,y_j),
\end{eqnarray*}
where $\phi_{\vec{\xi}}(\vec{x}) = \mathrm{e}^{\ii \vec{\xi}\cdot \vec{x}}$. 
Taking $\vec{f}=\vec{0}$ and inserting a plane wave $\vec{u} = \vec{X}\phi_{\vec{\xi}}(\vec{x})$ into \eqref{eq:MAC_FD_Grid_AT}, we obtain 
$$
\left( -\mu \sigma_{\Delta_h}\left(\vec{\xi} \right) I_d -(\lambda+\mu) \sigma_{\nabla_h}(\vec{\xi})\sigma_{\mathrm{div}_h}\left( \vec{\xi}\right)-\rho\omega^2 I_d \right)
\left(\begin{array}{c}
X_1 \phi_{\vec{\xi}}\left( x_i , y_{j+\frac{1}{2}} \right)
\\[0.5em]
X_2 \phi_{\vec{\xi}}\left( x_{i+\frac{1}{2}},y_j\right)
\end{array}\right) = \vec{0},
$$
with 
$$
\sigma_{\nabla_h}\left(\vec{\xi} \right) = \left(
\begin{array}{c}
\sigma_{\delta_x^-}\left(\vec{\xi} \right)
\\
\sigma_{\delta_y^-}\left(\vec{\xi} \right)
\end{array}
 \right),
\
 \sigma_{\mathrm{div}_h}\left(\vec{\xi} \right) =\left(\sigma_{\delta_x^+}\left(\vec{\xi} \right) ,\sigma_{\delta_y^+}\left(\vec{\xi} \right) \right),
$$
and 
$$
\sigma_{\Delta_h}\left(\vec{\xi} \right) = \sigma_{\mathrm{div}_h}\left(\vec{\xi} \right)\sigma_{\nabla_h}\left(\vec{\xi} \right).
$$
The symbol of the discretized elastic Helmholtz equation is thus the matrix
$$
\sigma_h(\vec{\xi}) = -\mu \sigma_{\Delta_h}(\vec{\xi}) I_d - (\lambda+\mu)\sigma_{\nabla_h}(\vec{\xi})\sigma_{\mathrm{div}_h}(\vec{\xi}) -\rho\omega^2 I_d,
$$
which is the equivalent of \eqref{eq:symb-cont} for the discretized problem. We can mimic the idea of the continuous case and using \eqref{eq:Formula_det}, with 
$$
A = \left(-\mu \sigma_{\Delta_h}(\vec{\xi})-\rho\omega^2\right) I_d,\ \vec{c} =  -(\lambda+\mu) \sigma_{\nabla_h}(\vec{\xi}),\ \vec{d}^T = \sigma_{\mathrm{div}_h}(\vec{\xi}),
$$
we obtain  
\begin{eqnarray*}
\mathrm{det}\,\sigma_h(\vec{\xi}) &=& \left(-\mu \sigma_{\Delta_h}(\vec{\xi})-\rho\omega^2\right) 
\left( -\mu \sigma_{\Delta_h}(\vec{\xi})-\rho\omega^2 -(\lambda+\mu)\sigma_{\mathrm{div}_h}(\vec{\xi}) \sigma_{\nabla_h}(\vec{\xi}) \right),
\\
&=& \left(-\mu \sigma_{\Delta_h}(\vec{\xi})-\rho\omega^2\right) 
\left( -(\lambda+2\mu) \sigma_{\Delta_h}(\vec{\xi})-\rho\omega^2 \right). 
\end{eqnarray*}
The discrete dispersion relation is therefore
\begin{equation}\label{eq:discrete_dispersion_relation-2D}
D_h := \left\{ \vec{\xi} \in \mathbb{R}^2 \ | \ \left(-\sigma_{\Delta_h}(\vec{\xi})-\omega^2\frac{\rho}{\mu}\right) 
\left( -\sigma_{\Delta_h}(\vec{\xi})-\omega^2\frac{\rho}{\lambda+2\mu} \right)=0 \right\},
\end{equation}
where $\sigma_{\Delta_h}(\vec{\xi})$ is the symbol of the discrete Laplace operator.

\begin{rem}[Dispersion analysis for the three-dimensional case]\label{rem:3d_MAC}
By a derivation similar to the two-dimensional case, the discrete symbol of the three-dimensional MAC scheme is
$$
\sigma_h(\vec{\xi}) = -\mu \sigma_{\Delta_h}(\vec{\xi}) I_d - (\lambda+\mu)\sigma_{\nabla_h}(\vec{\xi})\sigma_{\mathrm{div}_h}(\vec{\xi}) -\rho\omega^2 I_d,
$$
with 
$$
\sigma_{\nabla_h}\left(\vec{\xi} \right) = \left(
\begin{array}{c}
\sigma_{\delta_x^-}\left(\vec{\xi} \right)
\\
\sigma_{\delta_y^-}\left(\vec{\xi} \right)
\\
\sigma_{\delta_z^-}\left(\vec{\xi} \right)
\end{array}
 \right),
\
 \sigma_{\mathrm{div}_h}\left(\vec{\xi} \right) =\left(\sigma_{\delta_x^+}\left(\vec{\xi} \right) ,\sigma_{\delta_y^+}\left(\vec{\xi} \right) , \sigma_{\delta_z^+}\left(\vec{\xi} \right)\right)
$$
and thus formula \eqref{eq:discrete_dispersion_relation-2D} generalizes to:
\begin{equation}\label{eq:discrete_dispersion_relation}
D_h := \left\{ \vec{\xi} \in \mathbb{R}^d \ | \ \left(-\sigma_{\Delta_h}(\vec{\xi})-\omega^2\frac{\rho}{\mu}\right)^{d-1} 
\left( -\sigma_{\Delta_h}(\vec{\xi})-\omega^2\frac{\rho}{\lambda+2\mu} \right)=0 \right\}.
\end{equation}
\end{rem}

From \eqref{eq:discrete_dispersion_relation-2D} (or \eqref{eq:discrete_dispersion_relation} in $3$D), one has 
\begin{equation}\label{eq:Discrete_dispersion_relation}
D_h = \left\{-\sigma_{\Delta_h}(\vec{\xi}) - k_S^2=0 \right\} \bigcup \left\{-\sigma_{\Delta_h}(\vec{\xi}) - k_P^2=0 \right\}.
\end{equation}
We thus recover a discrete Helmholtz decomposition of the discrete plane wave, see \eqref{eq:discrete_Helmholtz_decomposition}, where each part is associated to a discrete Helmholtz equation with wavenumbers $k_S$ or $k_P$.   


As in the continuous case, we can define the discrete scalar wavenumber $\tilde{k}$ as 
$$
\forall \vec{\xi}\in D_h: \tilde{k} =\norme{\vec{\xi}}. 
$$
Writing any $\vec{\xi}\in D_h$ as $\vec{\xi} = \tilde{k} \vec{\xi}/\norme{\vec{\xi}}$ we can define the discrete wavenumber for each unit vector $\vec{\theta}\in \mathcal{S}^{d-1}$ (the unit circle or sphere) representing a direction of propagation, as the solution to
$$
-\sigma_{\Delta_h}(\tilde{k}\vec{\theta}) - k_S^2=0,\quad \text{or} \quad -\sigma_{\Delta_h}(\tilde{k}\vec{\theta}) - k_P^2=0,
$$
from which we clearly see that $\tilde{k}$ depends on the physical parameters $\lambda$, $\mu$, $\omega$ and $\rho$, the direction $\vec{\theta}$ and on the discretization parameter $h$. For the remainder of the manuscript, we will denote $\tilde{k}(\omega,\rho,\lambda,\mu,h,\vec{\theta})$ with $k \in \{k_S, k_P\}$. 
Considering now a uniform mesh with meshsize $h_x=h_y=h$, for the finite difference discretization described above, we have 
\begin{eqnarray*}
\sigma_{\Delta_h}(\vec{\xi}) &=& 
-\frac{4-2\left( \cos(\xi_1 h) + \cos(\xi_2 h)\right)}{h^2},\ \text{if}\ d=2,
\\[0.4em]
\sigma_{\Delta_h}(\vec{\xi}) &=&-\frac{6-2\left( \cos(\xi_1 h) + \cos(\xi_2 h) + \cos(\xi_3 h)\right)}{h^2},\ \text{if}\ d=3. 
\end{eqnarray*}
Using \cite[Theorem 3.2, Section 4.1, Section 4.3]{cocquet2024asymptotic}, we get the following expansion of the discrete wavenumber (see also \cite[Eq. (14), Eq. (17)]{cocquet2025improving}): as $kh\to 0$,
\begin{eqnarray*}
\tilde{k}(\omega,\rho,\lambda,\mu,h,\theta) &=& k  + \frac{k^3}{24}h^2 \left(\cos(\theta)^4 + \sin(\theta)^4 \right) + \cdots,\ \text{if}\ d=2,  
 \\ [0.4em]
\tilde{k}(\omega,\rho,\lambda,\mu,h,\theta,\varphi) &=& k +\frac{k^3}{24}h^2 \left(\cos(\phi)^4 \sin(\theta)^4 + \sin(\phi)^4\sin(\theta)^4 + \cos(\phi)^4  \right) + \cdots,\ \text{if}\ d=3,  
\end{eqnarray*}
and each formula can be adapted for each type of wave (P or S) 
simply by taking $k=k_S$ or $k=k_P$. 
More generally, the discrete wavenumber can be written as
\begin{equation}\label{eq:Err-disper}
\tilde{k}(\omega,\rho,\lambda,\mu,h,\vec{\theta}) = k +  \frac{k^3}{24}h^2 F(\vec{\theta}) + \dots,
\end{equation}
and thus the leading order of the relative dispersion error is 
$$
\text{Err}_{\text{disp}}  = \max_{\vec{\theta}\in \mathcal{S}^{d-1}} \left| \frac{\tilde{k}-k}{k} \right| = \frac{(kh)^2}{24}\max_{\vec{\theta}\in \mathcal{S}^{d-1}} |F(\vec{\theta})|. 
$$


\section{Asymptotic dispersion correction for the MAC scheme}
\label{sec:asymptotic}

We follow \cite{cocquet2021closed,cocquet2024asymptotic,cocquet2025improving} and perform asymptotic dispersion correction by introducing 
shifted/perturbed  
Lam\'e coefficients 
$\widehat{\lambda}$ and $\widehat{\mu}$ 
$$
\widehat{\lambda} = \lambda + h^2 \lambda_2,\ \widehat{\mu} = \mu + h^2 \mu_2.
$$

\begin{rem}
Shifting both $\lambda$ and $\mu$ allows us to correct the shear and pressure discrete wavenumbers simultaneously. 
This is particularly desired for compressible media, when the Poisson ratio is not too high.
In this case, 
$$
k_S = k_P\sqrt{2+\frac{\lambda}{\mu}} = k_P \sqrt{2 + \frac{2\nu}{1-2\nu}}\, ,
$$
so that 
$$
G_S =  G_P \sqrt{\frac{1-2\nu}{2(1-\nu)}},
$$
where 
$$
G_S =\frac{2\pi}{k_S h} , \ G_P =\frac{2\pi}{k_P h},
$$
are  the two different numbers of grid points per wavelength.
As a result, both wavenumbers and both numbers of grid point per wavelength can be in the same order of magnitude.
For example, one has $\sqrt{10}/2\approx 1.58<k_S/k_P = G_P/G_S<2$ for media with Poisson ratio $1/6\leq \nu < 1/3$, see \cite[Table 8]{gercek2007poisson}. 
We stress that our method is applicable regardless of the Poisson ratio. 
Nevertheless, in the nearly incompressible limit when $\nu\to0.5$, even for a low number of grid points per shear wavelength we have well-approximated pressure waves.
\end{rem}

To write the finite difference method using the shifted Lam\'e coefficients we use \eqref{eq:MAC_FD_Grid_AT} and replace each Lam\'e coefficient
by the shifted one. It reads
\begin{equation}\label{eq:Discrete_elastic_Helmholtz_shifted}
\left(-\widehat{\mu}\vec{\Delta}_h\,
- \left(\widehat{\lambda}+\widehat{\mu}\right)\nabla_h^- \mathrm{div}_h^+
- \rho\omega^2\right)\left(
\begin{array}{c}
u^x_{i+\frac{1}{2},j} \\ u^y_{i,j+\frac{1}{2}}
\end{array}
\right) = \left(
\begin{array}{c}
f^x_{i+\frac{1}{2},j} \\ f^y_{i,j+\frac{1}{2}}
\end{array}
\right),
\end{equation}
with a similar modification in $3$D (see Remark \ref{rem:3d_MAC}),
the dispersion relation associated to \eqref{eq:Discrete_elastic_Helmholtz_shifted} can be computed similarly to the one for 
\eqref{eq:Discrete_dispersion_relation}
and reads 
\begin{equation}\label{eq:Discrete_dispersion_relation_shift}
\begin{split}
D_h & = \left\{-\sigma_{\Delta_h}(\vec{\xi}) - \widehat{k}_S^2=0 \right\} \bigcup \left\{-\sigma_{\Delta_h}(\vec{\xi}) - \widehat{k}_P^2=0 \right\},\\
    & = \left\{-\sigma_{\Delta_h}(\vec{\theta}\widehat{\tilde{k}}) - \widehat{k}_S^2=0 \right\} \bigcup \left\{-\sigma_{\Delta_h}(\vec{\theta}\widehat{\tilde{k}}) - \widehat{k}_P^2=0 \right\},
\end{split}
\end{equation}
where 
$$
\widehat{k}_S(\omega,\rho,\lambda,\mu,h) =  \omega\sqrt{\frac{\rho}{\widehat{\mu}}}\quad \text{and}\quad \widehat{k}_P(\omega,\rho,\lambda,\mu,h) = \omega\sqrt{\frac{\rho}{\widehat{\lambda}+2\widehat{\mu}}}
$$
are the resulting shifted wavenumbers.
Expanding these discrete wavenumbers with shifted Lam\'e coefficient for small enough $h$ (and taking into account that $\widehat{k}_P$ and $\widehat{k}_S$ depend on $h$), yields 
\begin{eqnarray*}
\widehat{\tilde{k}}_P(\omega,\rho,\lambda,\mu,h,\vec{\theta}) &=& k_{P} +  h^2\left(\frac{k_P^3}{24}F(\vec{\theta}) +  k_P \frac{-2\mu_2 - \lambda_2}{4\mu + 2 \lambda}\right)+ \cdots
\\[0.4em]
\widehat{\tilde{k}}_{S}(\omega,\rho,\lambda,\mu,h,\vec{\theta})&=&  k_{S} +  h^2\left(\frac{k_S^3}{24} F(\vec{\theta}) - \frac{k_S}{2\mu}\mu_2 \right)+\cdots,
\end{eqnarray*} 
and we can now reduce the dispersion error for both $S$ and $P$ waves by computing $\lambda_2^{\text{asy}},\ \mu_2^{\text{asy}}$ so that 
\begin{equation}\label{eq:def_mu_asy}
\mu_2^{\text{asy}} = \argmin_{\mu_2}\left(\max_{\vec{\theta}\in \mathcal{S}^{d-1}}\left|F(\vec{\theta}) - \frac{24}{2\mu k_S^2}\mu_2 \right| \right)
\end{equation} 
and
\begin{equation}\label{eq:def_lambda_asy}
\lambda_2^{\text{asy}} = \argmin_{\lambda_2,\mu_2}\left(  \max_{\vec{\theta}\in 
\mathcal{S}^{d-1}}\left|F(\vec{\theta}) - 24\frac{2\mu_2 + \lambda_2}{(4\mu + 2 \lambda)k_P^2} \right| \right).
\end{equation}
The next theorem gives the optimal asymptotic values of $\lambda_2$ and $\mu_2$.
\begin{thm}\label{thm:Asymptotic_shift_2nd_order}
The unique solution to \eqref{eq:def_mu_asy}-\eqref{eq:def_lambda_asy} is 
\begin{eqnarray*}
&& \lambda_2^{\text{asy}}= -\frac{\rho\omega^2}{16},\,\mu_2^{\text{asy}} =\frac{\rho\omega^2}{16}\ \text{if}\ d=2, \ \text{and}
\\[0.4em]
&& \lambda_2^{\text{asy}}= -\frac{\rho\omega^2}{18},\,\mu_2^{\text{asy}} =\frac{\rho\omega^2}{18}\ \text{if}\ d=3.
\end{eqnarray*}
\end{thm}

\begin{proof}
Each of the optimization problems \eqref{eq:def_mu_asy} and \eqref{eq:def_lambda_asy} can be written as 
\begin{equation}\label{eq:general_optim_pbm}
\argmin_{t} \norme{z_1 F(\cdot) + t z_2}_{L^\infty(\mathcal{S}^{d-1})},
\end{equation}
for some continuous function $F$ and $z_1=z_2=1$. According to \cite[Proposition 4.1]{cocquet2025improving}, the unique solution to \eqref{eq:general_optim_pbm} is 
$$
t^* = - \frac{z_1}{z_2} \left(\frac{F_{\max}+F_{\min}}{2}\right),\ F_{\max} = \max_{\mathcal{S}^{d-1}}F,\ F_{\min} = \min_{\mathcal{S}^{d-1}} F, 
$$
and one has 
$$
\min_{t} \norme{z_1 F(\cdot) + t z_2}_{L^\infty(\Omega)} = |z_1|^2 \frac{|F_{\max}-F_{\min}|}{2}.
$$
As a result, taking $\lambda_2^{\text{asy}},\mu_2^{\text{asy}}$ so that 
$$
- 24\frac{2\mu_2^{\text{asy}} + \lambda_2^{\text{asy}}}{(4\mu + 2 \lambda)k_P^2} =  t^*,\ - \frac{24}{2\mu k_S^2}\mu_2^{\text{asy}} = t^*,
$$
we will get that 
\begin{eqnarray*}
\max_{\vec{\theta}\in 
\mathcal{S}^{d-1}}\left|F(\vec{\theta}) - 24\frac{2\mu_2\lambda_2^{\text{asy}} + \lambda_2\lambda_2^{\text{asy}}}{(4\mu + 2 \lambda)k_P^2} \right| &=& \max_{\vec{\theta}\in \mathcal{S}^{d-1}}\left|F(\vec{\theta}) - \frac{24}{2\mu k_S^2}\mu_2\lambda_2^{\text{asy}} \right|
\\
&=& 
\frac{|F_{\max}-F_{\min}|}{2},
\end{eqnarray*}
and thus both min-max problems \eqref{eq:def_lambda_asy}-\eqref{eq:def_mu_asy} are solved.

The values of $t^*$ for two dimensions and three dimensions are already given in \cite[Theorem 4.1, Theorem 4.3]{cocquet2024asymptotic} (see also \cite[Proposition 4.1]{cocquet2025improving}) and they are 
\begin{equation}\label{eq:asymp_k_Helmholtz}
t^* = \left\{\begin{array}{rl}
 \displaystyle -\frac{3}{4} & \text{if}\ d=2,
\\[1em]
 \displaystyle - \frac{2}{3} & \text{if}\ d=3.
\end{array}\right.
\end{equation}
This gives the aforementioned formula for the shifted Lam\'e coefficient.
\end{proof}

We show in Figure \ref{fig:disp_rel} the effect of the asymptotic optimal shifted Lamé coefficient on the dispersion relation.
In Figure \ref{fig:phase_err} we illustrate the expected dispersion error \eqref{eq:Err-disper} in logarithmic grid, calculated by sampling the dispersion relation for different numbers of grid points per shear wavelength. These numerical experiments show that the asymptotic dispersion correction reduces the dispersion error even for 
relatively small number of grid points per wavelength. 


\begin{figure}
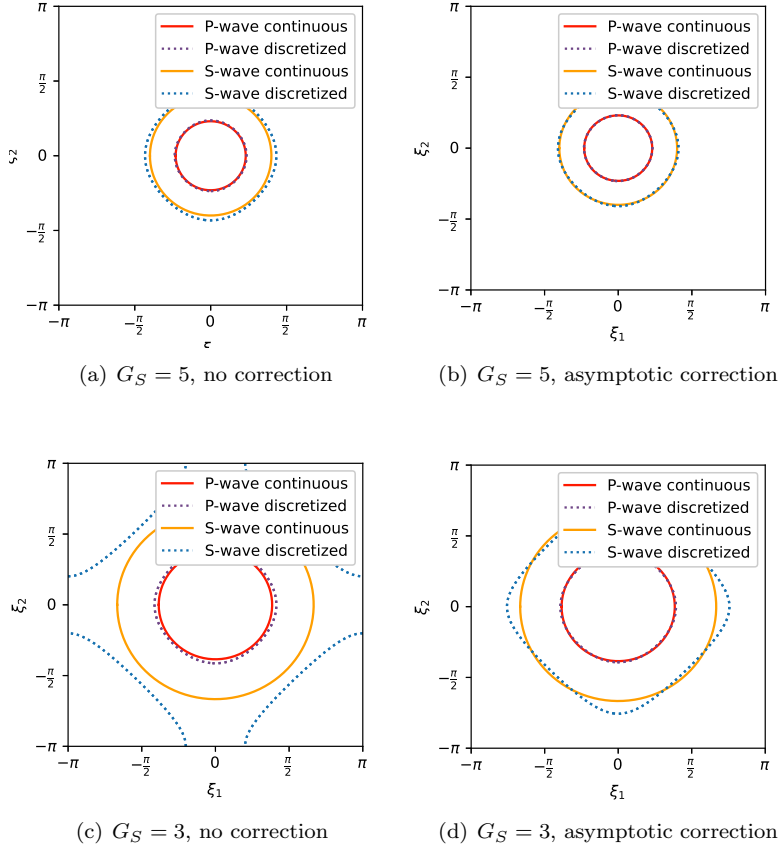

\begin{center}
	\newcommand{\image}[1]{\includegraphics[width=0.4\linewidth]{./#1}}
    \subfigure[\footnotesize $G_S=5$, no correction]{\image{disp_rel_5G_no_corr.eps}\label{fig:disp_rel_5G_no_corr}}
    \subfigure[\footnotesize $G_S=5$, asymptotic correction]{\image{disp_rel_5G_corr.eps}\label{fig:disp_rel_5G_corr}}
    \subfigure[\footnotesize $G_S = 3$, no correction
]{\image{disp_rel_3G_no_corr.eps}\label{fig:disp_rel_3G_no_corr}}
    \subfigure[\footnotesize $G_S = 3$, asymptotic correction]{\image{disp_rel_3G_corr.eps}\label{fig:disp_rel_3G_corr}}
\\
\end{center}
\caption{Dispersion relations for the standard second-order MAC scheme for the elastic Helmholtz equation, with and without dispersion correction. We used Poisson ratio $\nu = 0.25$ so that $G_P = \sqrt{3} G_S$. 
}\label{fig:disp_rel}
\end{figure}

\begin{figure}
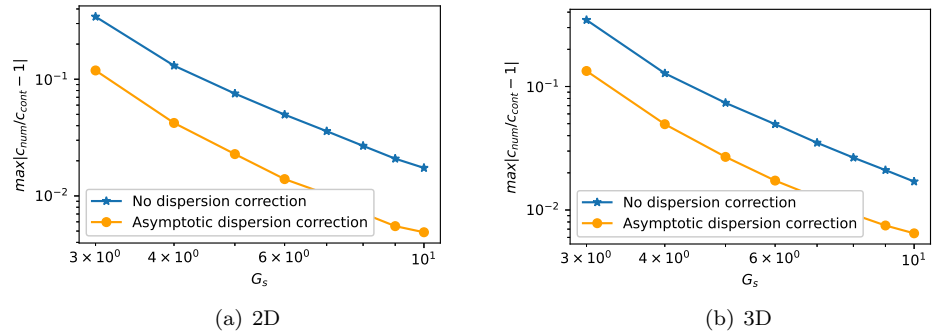

\begin{center}
	\newcommand{\image}[1]{\includegraphics[width=0.49\linewidth]{./#1}}
   \subfigure[\footnotesize 2D]{\image{phase_err.eps}\label{fig:phase_err_2D}}
  \subfigure[\footnotesize 3D]{\image{phase_err_3D.eps}\label{fig:phase_err_3D}}\\
\end{center}
\caption{Dispersion error calculated by sampling the symbol for different $G_S$. On the left for two dimensions and on the right for three dimensions.
}\label{fig:phase_err}
\end{figure}

\section{Extension to other FD scheme and limits}
\label{sec:extension}

We provide here a direct extension of the computations of the asymptotic optimal Lamé parameters for general FD schemes for which the scalar Laplacian is approximated as $\mathrm{div}_h\nabla_h$. We then give some limitations of the approach presented in this paper.

\subsection{A general recipe for asymptotic dispersion correction using shifted Lamé coefficients}\label{sec:Extension_shifted_LAme}

Let us consider a FD discretization such that 
\begin{equation}\label{eq:Assumption_Delta_h=div_grad}
\Delta_h=\mathrm{div}_h\nabla_h,
\end{equation}
and possibly involving a mass-matrix in front of the term $\omega^2\rho$ (see e.g. \cite{yovel2024lfa}). In this case, the discrete symbol can be expressed as 
$$
\sigma_h(\vec{\xi}) = -\mu \sigma_{\Delta_h}(\vec{\xi}) I_d - (\lambda+\mu)\sigma_{\nabla_h}(\vec{\xi})\sigma_{\mathrm{div}_h}(\vec{\xi}) -\rho\omega^2 \sigma_{M_h}(\vec{\xi}) I_d, 
$$
where $\sigma_{M_h}(\vec{\xi})$ is the symbol of the mass-matrix which approximates the identity operator as $h\to0$. Using \eqref{eq:Formula_det}, the dispersion relation is
\begin{equation}\label{eq:discrete_dispersion_relation_general}
D_h = D_h(k_S)\cup D_h(k_P),
\end{equation}
where
\begin{equation}\label{eq:dispersion_relation_Helmholtz}
D_h(k) = \left\{ -\sigma_{\Delta_h}(\vec{\xi})-k^2\sigma_{M_h}(\vec{\xi})=0\right\}. 
\end{equation}
We emphasize that $D_h(k)$ is exactly the discrete dispersion relation of the (scalar) Helmholtz equation $-\Delta u - k^2u =f$ discretized (in $3$D) with
$$
-\Delta_h u_{i,j,k} - k^2 M_h u_{i,j,k} = f_{i,j,k}. 
$$
The discrete wavenumber $\tilde{k}$ can be defined as above as 
$$
\forall \vec{\theta}\in \mathcal{S}^{d-1}:\  -\sigma_{\Delta_h}(\tilde{k}\vec{\theta})-k^2\sigma_{M_h}(\tilde{k} \vec{\theta})=0,
$$
and we provided in \cite[Proposition 3.1, Theorem 3.2, Theorem A.1]{cocquet2024asymptotic} some assumptions\footnote{The latter mostly require the FD stencil to be of order $p$ on plane waves.} under which one has 
\begin{equation}\label{eq:expansion_k_d_Helmholtz}
\tilde{k}(\omega, \rho, \lambda, \mu,h,\vec{\theta}) = k + h^p F_p(k,\vec{\theta}) + \cdots,
\end{equation}
for some known function $F_p$ which is smooth in each of its variables. Using this and \eqref{eq:discrete_dispersion_relation_general}, the two discrete wavenumbers 
satisfy 
\begin{eqnarray*}
\tilde{k}_{P}(\omega,\rho,\lambda,\mu,h,\vec{\theta}) &=& k_P + h^p F_p(k_P,\vec{\theta}) + \cdots,
\\
\tilde{k}_{S}(\omega,\rho,\lambda,\mu,h,\vec{\theta}) &=& k_S + h^p F_p(k_S,\vec{\theta}) + \cdots.
\end{eqnarray*}
Introducing the shifted Lam\'e parameters 
$$
\widehat{\lambda} = \lambda + h^p \lambda_p,\ \widehat{\mu} = \mu + h^p \mu_p,
$$
in the FD approximation, results in having the following expansions for the associated discrete wavenumbers
\begin{eqnarray*}
\widehat{\tilde{k}}_{P}(\omega,\rho,\lambda,\mu,h,\vec{\theta}) &=& \widehat{k}_P + h^p F_p(\widehat{k}_P,\vec{\theta}) + \cdots,
\\
\widehat{\tilde{k}}_{S}(\omega,\rho,\lambda,\mu,h,\vec{\theta}) &=& \widehat{k}_S + h^p F_p(\widehat{k}_S,\vec{\theta}) + \cdots.
\end{eqnarray*}
Using the expansions
\begin{eqnarray*}
\widehat{k}_S  &=& \omega \sqrt{\frac{\rho}{\widehat{\mu}}} =  k_S - h^p\frac{k_S}{2\mu}  \mu_p + \cdots
\\
\widehat{k}_P  &=&\omega \sqrt{\frac{\rho}{\widehat{\lambda}+2\widehat{\mu}}} =  k_P - h^p\frac{k_P}{2(\lambda + 2\mu)} \left(\lambda_p + 2\mu_p \right) + \cdots,
\end{eqnarray*}
both the discrete wavenumbers associated to the FD scheme with shifted Lam\'e coefficients satisfy
\begin{eqnarray*}
\widehat{k}_{d,S}(\omega,\rho,\lambda,\mu,h,\vec{\theta}) &=& k_S + h^p \left(F_p(k_S,\vec{\theta}) -\frac{k_S}{2\mu}  \mu_p\right) + \cdots,
\\
\widehat{k}_{d,P}(\omega,\rho,\lambda,\mu,h,\vec{\theta}) &=& k_P + h^p \left(F_p(k_P,\vec{\theta}) -\frac{k_P}{2(\lambda + 2\mu)} \left(\lambda_p + 2\mu_p \right)\right)+ \cdots.
\end{eqnarray*}
We can now define the asymptotically optimal shifts that minimize the dispersion errors for both the $S$ and $P$ waves as 
\begin{equation}\label{eq:Asympt_optimal_Lambda_shifts}
\lambda_p^{\text{asy}} = \argmin_{\lambda_p} \left(\max_{\vec{\theta}}\left|F_p(k_P,\vec{\theta}) -\frac{k_S}{2\mu}  \mu_p\right| 
\right),
\end{equation}
and
\begin{equation}\label{eq:Asympt_optimal_mu_shifts}
\mu_p^{\text{asy}} = \argmin_{\mu_p} \left(\max_{\vec{\theta}}\left|F_p(k_S,\vec{\theta}) -\frac{k_P}{2(\lambda + 2\mu)} \left(\lambda_p + 2\mu_p \right)\right|\right).
\end{equation}
Using \cite[Proposition 4.1]{cocquet2025improving}, the unique solution to 
$$
\argmin_{t} \max_{\vec{\theta}} \left|F_p(k,\vec{\theta}) + t  \right|,
$$
is 
$$
t^*(k) = - \frac{F_{\max}(k) + F_{\min}(k)}{2},\ F_{\max}(k) = \max_{\vec{\theta}} F(k,\vec{\theta}) ,\  F_{\min}(k)=\min_{\vec{\theta}} F(k,\vec{\theta}),
$$
and one has 
$$
\min_{t} \max_{\vec{\theta}} \left|F_p(k,\vec{\theta}) + t  \right| = \frac{|F_{\max}(k)-F_{\min}(k)|}{2}. 
$$
The unique solution to \eqref{eq:Asympt_optimal_Lambda_shifts} and \eqref{eq:Asympt_optimal_mu_shifts} can thus be computed from  
\begin{equation}\label{eq:asympt_optimal_shifts_general}
 -\frac{k_S}{2\mu}  \mu_p = t^*(k_S),
\
-\frac{k_P}{2(\lambda + 2\mu)} \left(\lambda_p + 2\mu_p \right) = t^*(k_P).
\end{equation}
A recipe to get the asymptotic optimal Lam\'e coefficients can finally be summarized as:
\begin{enumerate}
\item Compute an expansion of the discrete wavenumber associated to the elastic Helmholtz equation with wavenumber $k$ and identify the function $F_p$ in its leading order term (see \eqref{eq:expansion_k_d_Helmholtz}).
\item Compute the minimal and maximal values of $F_p$ over the unit sphere.
\item Compute the asymptotic optimal shifts from \eqref{eq:asympt_optimal_shifts_general}. 
\end{enumerate}

\subsection{Limitations of the general recipe}

We give here some limitations of dispersion correction using the shifted Lamé coefficients as described in Section \ref{sec:Extension_shifted_LAme}. 




\subsection*{FD scheme whose dispersion relation can be factorized}
We emphasize that our dispersion correction technique is based on the fact that the discrete dispersion relation can be factorized as the continuous one. This factorization can actually been done even for FD scheme such that \eqref{eq:Assumption_Delta_h=div_grad} is not satisfied. Let us consider FD operators $\Delta_h$, $\nabla_h$ and $\mathrm{div}_h$ such that  
$$
\Delta_h \neq \mathrm{div}_h \nabla_h,
$$
and approximate \eqref{eq:Elastic_Helmholtz_homogeneous_mu} directly as $\left(-\mu\Delta_hI_d - (\lambda+\mu)\nabla_h\mathrm{div}_h-\rho\omega^2 M_h I_d\right)$ where $M_h$ is a mass-matrix. Note that the discrete symbols are such that
$$
\sigma_{\Delta_h}\neq \sigma_{\mathrm{div}_h}\sigma_{\nabla_h},
$$ 
and the discrete symbol is now
$$
\sigma_h(\vec{\xi}) =  -\mu \sigma_{\Delta_h}(\vec{\xi}) I_d - (\lambda+\mu)\sigma_{\nabla_h}(\vec{\xi})\sigma_{\mathrm{div}_h}(\vec{\xi}) -\rho\omega^2 \sigma_{M_h}(\vec{\xi}) I_d.
$$
Using then \eqref{eq:Formula_det}, one has 
$$
\mathrm{det}\left(\sigma_h(\vec{\xi}) \right)
= \sigma_S(\vec{\xi})\, \sigma_P(\vec{\xi})
$$
where
$$
\sigma_S(\vec{\xi})
= \left(-\mu\sigma_{\Delta_h}(\vec{\xi})-\omega^2\rho\sigma_{M_h}(\vec{\xi})\right)^{d-1}
$$
and
$$
\sigma_P(\vec{\xi})
= \left(-\mu\sigma_{\Delta_h}(\vec{\xi})-\omega^2\rho\sigma_{M_h}(\vec{\xi}) 
-(\lambda+\mu) \sigma_{\mathrm{div}_h}(\vec{\xi}) \sigma_{\nabla_h}(\vec{\xi})
\right). 
$$
Note that $\sigma_{\mathrm{div}_h}$ and $\sigma_{\nabla_h}$ approximate the symbols of the corresponding continuous operators hence, as $h\to 0$, one has 
$$
\sigma_{\nabla_h}(\vec{\xi}) = \ii \vec{\xi} + O(h),\ \sigma_{\mathrm{div}}(\vec{\xi}) = \ii \vec{\xi}^T + O(h), 
$$
and thus $\sigma_{\mathrm{div}_h}(\vec{\xi}) \sigma_{\nabla_h}(\vec{\xi})=-\norme{\vec{\xi}}^2+O(h)$ so that it is indeed an approximation of the Laplacian. The P and S waves are therefore no longer approximated with the same finite-differences Laplacian and thus the general recipe given in Section \ref{sec:Extension_shifted_LAme} has to be extended to such FD methods. Although the recipe can be adapted to this new case, it is based on an explicit expansion of the discrete wavenumber (see \eqref{eq:expansion_k_d_Helmholtz}). The latter can be obtained from \cite[Proposition 3.1, Theorem 3.2, Theorem A.1]{cocquet2024asymptotic} which is specific to discretization of Helmholtz equations which involves the same FD-Laplacian hence being the main reason we cannot readily extend the recipe to such FD schemes. 

\subsection*{A FD scheme for which $(\nabla \mathrm{div})_h \neq \nabla_h\mathrm{div}_h$ }

There also exists some FD schemes for which \eqref{eq:Assumption_Delta_h=div_grad} is not satisfied. 
Such examples can be found in \cite{sjogreen2012fourth,nilsson2007stable}. More precisely, the second-order FD scheme from \cite{nilsson2007stable} approximates \eqref{eq:Elastic_Helmholtz_homogeneous_mu} as (see \cite[Eq. (17)-(18)]{nilsson2007stable} and specify to two-dimensional setting and constant coefficients)
\begin{equation}\label{eq:Nilsson_2nd_order_FD}
\begin{array}{rcl}
-\mu \Delta_h u^x_{i,j} -(\lambda+\mu) \left(\delta_x^+\delta_x^- u^x_{i,j} + \delta_{xy}^hu^y_{i,j}\right)-\rho\omega^2 u^x_{i,j}
&=& f_{i,j}^x,
\\[0.5em]
-\mu \Delta_h u^y_{i,j} 
-(\lambda+\mu) \left(\delta_y^+\delta_y^- u^y_{i,j} + \delta_{xy}^hu^x_{i,j}\right)-\rho\omega^2 u^y_{i,j}
&=& f^y_{i,j},
\end{array}
\end{equation}
where the approximation of the cross-derivative is 
$$
\delta_{xy}^h v_{i,j} = \frac{1}{4h^2}\left(v_{i+1,j+1}-v_{i-1,j+1} - v_{i+1,j-1}+v_{i-1,j-1} \right).
$$
The FD scheme \eqref{eq:Nilsson_2nd_order_FD} can be written in compact form as 
$$
\left(-\mu I_d\Delta_h - (\lambda+\mu)\left(\nabla \mathrm{div} \right)_h -\rho\omega^2 I_d\right) \left( 
\begin{array}{c}
u^x_{i,j} \\ u^y_{i,j}
\end{array}
\right)
=  \left( 
\begin{array}{c}
f^x_{i,j} \\ f^y_{i,j}
\end{array}
\right).
$$
Note that the FD scheme \eqref{eq:Nilsson_2nd_order_FD} has a second-order truncation error. In addition, one has 
$$
\delta_{xy}^h v_{i,j} = \frac{1}{4}\left( \delta_{x}^+ + \delta_x^- \right)\left( \delta_{y}^+ + \delta_y^- \right)v_{i,j},
$$
so that the discrete symbol is 
$$
\sigma(\vec{\xi}) = (-\mu \sigma_{\Delta_h}(\vec{\xi}) - \rho \omega^2) I_d - 
(\lambda+\mu) \sigma_{(\grad\mathrm{div})_h}(\vec{\xi}),
$$
where 
$$
\sigma_{(\grad\mathrm{div})_h}(\vec{\xi}) =\frac{1}{h^2}
\left(
\begin{array}{cc}
-2(1 - \cos(\xi_x h)) &  -\sin(\xi_x h)\sin(\xi_y h)
\\
-\sin(\xi_x h)\sin(\xi_y h) & -2(1 - \cos(\xi_y h))
\end{array}
\right). 
$$
We have the following result. 
\begin{lemma}\label{lemma:Non_factorization}
Let $\vec{\xi}\notin (2\pi/h) \mathbb{Z}^2$, then there does not exist $\vec{c},\vec{d}\in \mathbb{C}^2$ such that 
$$
\sigma_{(\grad\mathrm{div})_h}(\vec{\xi}) = \vec{c}^T\vec{d}. 
$$
\end{lemma}

\begin{proof}
We assume there exists $\vec{c},\vec{d}\in \mathbb{C}^2$ (that may depend on $\vec{\xi}$) such that $
\sigma_{(\grad\mathrm{div})_h}(\vec{\xi}) = \vec{c}^T\vec{d}$. This gives $\mathrm{det}\left( \sigma_{(\grad\mathrm{div})_h}(\vec{\xi})\right)=0$ and we are going to show that this can only occurs if $\vec{\xi}\in (2\pi/h) \mathbb{Z}^2$.

We have 
$$
\mathrm{det}\left(\sigma_{(\grad\mathrm{div})_h}(\vec{\xi}) 
\right) = \frac{1}{h^4}\left(4(1-\cos(\xi_x h))(1-\cos(\xi_y h)) - \sin(\xi_x h)^2\sin(\xi_y h)^2   \right).
$$
Setting $X=\cos(\xi_xh)$ and $Y=\cos(\xi_yh)$, the solution to $\mathrm{det}\left( \sigma_{(\grad\mathrm{div})_h}(\vec{\xi})\right) =0$ satisfies
$$
0=4(1-X)(1-Y) - (1-X^2)(1-Y^2) =  (X-1)(Y-1)(XY+X+Y-3).
$$
Note that $XY+X+Y-3 = (1+X)(1+Y)-4 \leq 0$ with the maximal value reached at $X=Y=1$. As a result, the only solutions are 
$X=Y = 1$, which translates to $\vec{\xi} \in   (2\pi/h)\mathbb{Z}^2$. 
\end{proof}

Lemma \ref{lemma:Non_factorization} ensures that we cannot factorize $\sigma_{(\grad\mathrm{div})_h}(\vec{\xi})$ as $\vec{c}^T\vec{d}$ unless $h\vec{\xi}\in 2\pi \mathbb{Z}^2$ but in that case $\sigma_{(\grad\mathrm{div})_h}(\vec{\xi})=O$. As a result, there do not exists FD operators $\nabla_h$ and $\mathrm{div}_h$, approximating $\nabla$ and $\mathrm{div}$, such that $(\nabla \mathrm{div})_h = \nabla_h \mathrm{div}_h$ for the FD scheme \eqref{eq:Nilsson_2nd_order_FD}.

Nevertheless, we can still compute the dispersion relation in a factorized form. Indeed, we can remark that the matrix $\sigma_{(\grad\mathrm{div})_h}$ is symmetric and therefore diagonalizable. Denoting by $r_1(\vec{\xi})$ and $r_2(\vec{\xi})$ its two eigenvalues, we easily check that:
\begin{equation}\label{eq:facto-symbol-r1r2}
    \det(\sigma) = \left( -\mu \sigma_{\Delta_h} - \rho \omega^2 - (\lambda+\mu) r_1 \right) \left(-\mu \sigma_{\Delta_h} - \rho \omega^2 - (\lambda+\mu) r_2) \right).
\end{equation}
Moreover, we have (after some computations) that
\[
    r_1(\vec{\xi}) = \dfrac{-3+ \cos(\xi_x h) + \cos(\xi_y h) + \cos(\xi_x h ) \cos(\xi_y h)}{h^2} = -\| \vec{\xi} \|^2 + O(h^2)
\]
and 
\[
    r_2(\vec{\xi}) = -\dfrac{ (\cos(\xi_x h) -1) (\cos(\xi_y h) -1)}{h^2} = -\dfrac{\xi_x \xi_y}{4}h^2 + O(h^4).
\]
This shows that in the factorization of the symbol \eqref{eq:facto-symbol-r1r2} the first part 
$$
    -\mu \sigma_{\Delta_h} - \rho \omega^2 - (\lambda+\mu) r_1  =  -(\lambda + 2\mu) \sigma_{\Delta_h} - \rho \omega^2 + O(h^2) 
$$ 
corresponds to P-waves while the second part 
$$
    -\mu \sigma_{\Delta_h} - \rho \omega^2 - (\lambda+\mu) r_2 =  -\mu \sigma_{\Delta_h} - \rho \omega^2 + O(h^2)
$$
to S-waves. Now, noticing that 
\begin{eqnarray*}
 r_1 &=& \sigma(R_1^h),\ \text{with}\ 
 R_1^h v_{i,j} =  \left(\delta_{xx}^h+\delta_{yy}^h + \frac{h^2}{4}\delta^h_{xx}\delta^h_{yy}\right) v_{i,j},
 \\
 r_2 &=& \sigma(R_2^h),\ \text{with}\ R_2^h v_{i,j} =  - \frac{h^2}{4}\delta^h_{xx}\delta^h_{yy} v_{i,j},
\end{eqnarray*}
we can go back to the stencil notations and deduce the two discretization of the scalar Helmholtz equations (with different wavenumbers) corresponding to each part of the symbol. As a result, the discrete S and P waves associated to this FD scheme satisfy the following discrete Helmholtz equations 
\begin{equation}
\begin{split}
\text{P-waves}: \;\; & \left( -(\lambda+2\mu) \Delta_h - \rho \omega^2 - (\lambda+\mu)\frac{h^2}{4}\delta^h_{xx}\delta^h_{yy}  \right),
\\
\text{S-waves}: \;\; & \left(-\mu \Delta_h - \rho \omega^2 + (\lambda+\mu) \frac{h^2}{4}\delta^h_{xx}\delta^h_{yy} \right).
\end{split}
\end{equation}
In this form, it would also be possible to perform dispersion correction for each type of waves. Nevertheless, we did not focus on this part in this work for two main reasons: first numerical tests show that the discretization scheme \eqref{eq:Nilsson_2nd_order_FD} leads to larger error than the MAC sheme, and second we cannot directly rely on previous results on DC for Helmholtz equation to compute the corrected parameters $\widehat{\lambda}$ and $\widehat{\mu}$.

\section{Numerical experiments}\label{sec:Numerics}

To conclude this paper, let us show some numerical experiments and illustrate the impact of the dispersion correction on the MAC scheme. We consider the following problem posed in a bounded domain $\Omega \subset \mathbb{R}^d$:
\begin{equation}\label{eq:Elastic_Helmholtz_inhomogeneous_Lame}
-\textbf{div}\left\{\mu \left(\nabla \vec{u}+ \nabla\vec{u}^T \right)\right\} - \nabla \lambda\mathrm{div}\,\vec{u} 
- \rho\omega^2 \vec{u} = \vec{f}\quad \text{in}\quad \Omega.
\end{equation}
The PDE is equipped with boundary conditions, typically absorbing boundary conditions \cite{engquist1977absorbing,mattesi2019high} or perfectly matched layer \cite{berenger1994perfectly} to model an unbounded domain by minimizaing the reflections from the boundaries. 
\begin{rem}
Equipped with boundary conditions, the existence and uniqueness of a solution to the problem \eqref{eq:Elastic_Helmholtz_inhomogeneous_Lame} can be proved mostly using Fredholm alternative (see e.g. \cite[Section 5]{cocquet2012existence}).
\end{rem}

\subsection{Comparison with analytical solutions}\label{subsec:comp-analytical}
In this first part, we consider the two situations with an analytical solutions. First, we use the manufactured solution 
\begin{equation}\label{sol_analy-1}
\bfu(\vec x) = \begin{pmatrix}(\sin\left(\frac{k_S}{\sqrt{2}} x\right)\cos\left(\frac{k_S}{\sqrt{2}} y\right)\\ -\cos\left(\frac{k_S}{\sqrt{2}} x\right) \sin\left(\frac{k_S}{\sqrt{2}} y\right)\end{pmatrix}
\end{equation}
in two dimensions, and 
\begin{equation}\label{sol_analy-1_3D}
\bfu(\vec x)=
\left(
\begin{array}{c}
k_S \cos(k_S x)\left(-\sin(k_S y)\cos(k_S z)+\cos(k_S y)\sin(k_S z)\right)\\[1ex]
k_S \cos(k_S y)\left(-\sin(k_S z)\cos(k_S x)+\cos(k_S z)\sin(k_S x)\right)\\[1ex]
k_S \cos(k_S z)\left(-\sin(k_S x)\cos(k_S y)+\cos(k_S x)\sin(k_S y)\right)
\end{array}
\right)
\end{equation}
in three dimensions,
where $k_S$ is determined by the angular frequency $\omega$. This $\bfu$ solves the homogeneous elastic Helmholtz equation with null source term. In our numerical tests, we take $\lambda = 1$, $\mu=2$ and the density $\rho=1$, with Dirichlet boundary conditions and the values of $\bfu$ on the boundaries. 

In Figure \ref{fig:OOA}, we compare to the exact solution to the numerical one and present the truncation errors with and without dispersion correction for three grid sizes.

We compare to the exact solution and present the truncation errors with and without dispersion correction for three grid sizes.
We observe that the order of accuracy is two, since the resulting curves are parallel to a line with slope 2.
We also observe a significant reduction in the truncation error. 
We note that this reduction is present at a relatively large frequency (for the chosen grid sizes), since the dispersion correction is most efficient when the manufactured solution is highly oscillatory.
We note that there is a nice agreement of trends between Figure \ref{fig:phase_err} in two dimensions and in thee dimensions, showing the theoretically expected errors, and Figure \ref{fig:OOA}, showing the errors in practice for an oscillatory manufactured solution in two dimensions.

\begin{figure}
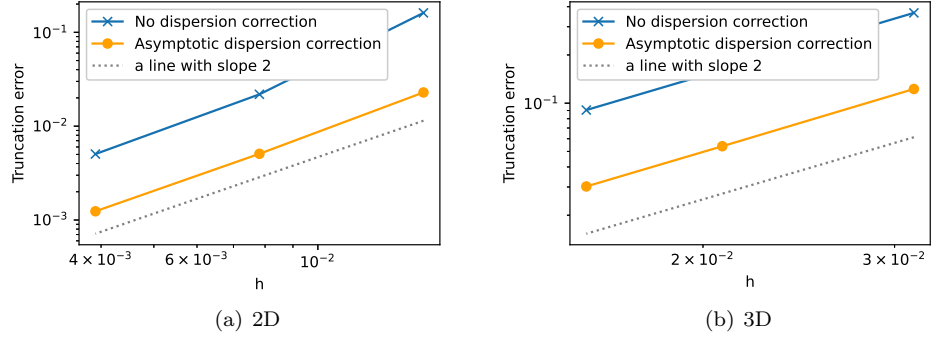

\begin{center}
	\newcommand{\image}[1]{\includegraphics[width=0.49\linewidth]{./#1}}
\subfigure[\footnotesize 2D]{\image{OOA.eps}\label{fig:OOA_2D}}
\subfigure[\footnotesize 3D]{\image{OOA_3D.eps}\label{fig:OOA_3D}}\\
\end{center}
\caption{Errors for three different grid sizes with and without dispersion correction for the constant Lam{\'e} coefficients and density $\lambda=\mu=\rho=1$, presented in logarithmic scale.
On the left, for two dimensions, we use $\omega =15\pi$, grid sizes $64^2,128^2$ and $256^2$ cells, and the manufactured solution \eqref{sol_analy-1}.
On the right, for three dimensions, we use $\omega =5\pi$, grid sizes $32^3,48^3$ and $64^3$ cells, and the manufactured solution \eqref{sol_analy-1_3D}.
}\label{fig:OOA}
\end{figure}

To quantitatively assess the reduction in the truncation error for different frequencies, we present in Table \ref{tab:OOA} the ratio of the truncation error with and without dispersion correction for different numbers of grid points per wavelength and different frequencies $\omega$. As we can see, dispersion correction allows one to reduce the error by a factor close to $2$ for large $G_S$. Also, one can notice that it is more interesting for large frequencies.

\begin{table}
\centering
\begin{tabular}{|l|l|l|l|l|l|}
  \hline
  \multicolumn{6}{|c|}{$ \mathbf{\omega = 20}$} \\
  \hline
  $G_S$ & $4.78$ & 	$9.91$ & 	$14.70$ & 	$19.82$ & 	$29.73$ \\ 
  \hline
  Err. No DC & $0.9514$ &	$0.1702$ &	$0.0778$ &	$0.0433$ &	$0.0194$ \\ 
  Err. With DC & $0.3957$ &	$0.0850$ &	$0.0389$ &	$0.0217$ &	$0.0097$ \\ 
  \hline
  Ratio & $\mathbf{2.4044}$ &	$\mathbf{2.0024}$ &	$\mathbf{1.9991}$ &	$\mathbf{1.9984}$ &	$\mathbf{1.9983}$ \\ 
  \hline
  \hline 
  \multicolumn{6}{|c|}{$ \mathbf{\omega = 50}$} \\
  \hline
  $G_S$ & $4.92$ & 	$9.98$ & 	$14.90$ & 	$19.96$ & 	$29.94$ \\ 
  \hline
  Err. No DC & $1.9833$ &	$0.8065$ &	$0.7750$ &	$0.1596$ &	$0.0562$\\ 
  Err. With DC & $0.5413$ &	$0.1661$ &	$0.0848$ &	$0.0503$ &	$0.0237$ \\ 
  \hline
  Ratio & $\mathbf{3.6641}$ &	$\mathbf{4.8547}$ &	$\mathbf{9.1430}$ &	$\mathbf{3.1715}$ &	$\mathbf{2.3698}$ \\ 
  \hline
  \hline
  \multicolumn{6}{|c|}{$ \mathbf{\omega = 100}$} \\
  \hline
  $G_S$ & $4.99$ & 	$9.98$ & 	$14.97$ & 	$19.96$ & 	$29.94$  \\ 
  \hline
  Err. No DC & $37.0736$ &	$1.9183$ &	$3.7506$ &	$0.4423$ &	$0.1318$\\ 
  Err. With DC & $2.3817$ &	$0.4149$ &	$0.1976$ &	$0.1176$ &	$0.0551$\\ 
  \hline
  Ratio & $\mathbf{15.5659}$ &	$\mathbf{4.6235}$ &	$\mathbf{18.9848}$ &	$\mathbf{3.7601}$ &	$\mathbf{2.3907}$ \\ 
  \hline
\end{tabular}
\caption{Error for various $\omega \in \{10, 50,100\}$ with or without dispersion correction (the error is truncated after $10^{-5}$). The ratio between the error with and without dispersion correction is calculated as $\text{Error without DC} / \text{Error with DC}$, so that values $>1$ mean that the error with dispersion correction is lower that the one without (and are marked in bold).}
\label{tab:OOA}
\end{table}

Let us now show a second case where we have an analytical solution given by the Green function of the PDE. We consider the particular situation of a circular S-waves. More precisely, we recall that in $\mathbb{R}^2\setminus \{ \mathbf{0}\}$, the function:
\begin{equation}
    \mathbf{u}_{ex}(x,y) = \dfrac{H'(k_S r)}{r}\left[ \begin{matrix}
        -y \\ x
    \end{matrix} \right]  
\end{equation}
where $r = \sqrt{x^2 + y^2}$ and $H$ is the Hankel function of the first kind, is solution to the homogeneous equations \eqref{eq:split}. Therefore, the following problem: Find $\mathbf{u}\in \left[H^1(\Omega)\right]^2$ satisfying
\begin{equation}\label{pb:Swaves}
\left\{\begin{array}{lcl}
     -\mu \mathbf{\Delta} \mathbf{u} - (\lambda+\mu) \nabla \text{div} (\mathbf{u}) - \rho \omega^2 \mathbf{u} = \mathbf{0} & \text{in} & \Omega \\[5pt]
     \sigma(\mathbf{u})\nu + \ii \boldsymbol{\gamma}(\boldsymbol{n}) \; \mathbf{u}  = \sigma(\mathbf{u}_{ex})\nu + \ii \boldsymbol{\gamma}(\boldsymbol{n}) \; \mathbf{u}_{ex}  &  \text{on} &  \partial \Omega. 
\end{array}   \right.
\end{equation}
has $\mathbf{u} = \mathbf{u}_{ex}$ as (unique) solution. For the boundary conditions (BC), we consider Robin type BCs which model Absorbing BCs (``Kupradze ABC''). In that case, this amounts to take for the matrix $\boldsymbol{\gamma}$:
\[
    \boldsymbol{\gamma}(\boldsymbol{n}) = \left[ \begin{matrix} \omega (\sqrt{\lambda+2\mu} \, n_1 + \sqrt{\mu} \, n_2) & 0 \\ 0 &  \omega (\sqrt{\mu} \, n_1 + \sqrt{\lambda+2\mu} \, n_2) \end{matrix} \right]
\]
where $\boldsymbol{n} = (n_1, n_2)$ is the outward normal. 

In our numerical tests, we take $\Omega = [-1,1]^2\setminus [-0.125,0.125]^2$ so that we avoid dealing with the singularity of the Hankel at the origin. For the physical parameters, we take $\rho = 1$, $\lambda = 16$, $\mu = 2$ and $\omega = 50$. We have represented in Figure \ref{fig:illus-DC-S-Waves} the numerical solution to \eqref{pb:Swaves} using DC or not, for two values of the discretization step. Also, in Figure \ref{fig:illus-DC-S-Waves-conv}, we have represented the convergence curve of the error versus the discretization step, again with or without DC. In both cases, we can see the interest of the DC.  
\begin{figure}[h]
\begin{center}
	\newcommand{\image}[1]{\includegraphics[width=0.45\linewidth]{./#1}}
    \subfigure[\footnotesize Case $n=30$]{\image{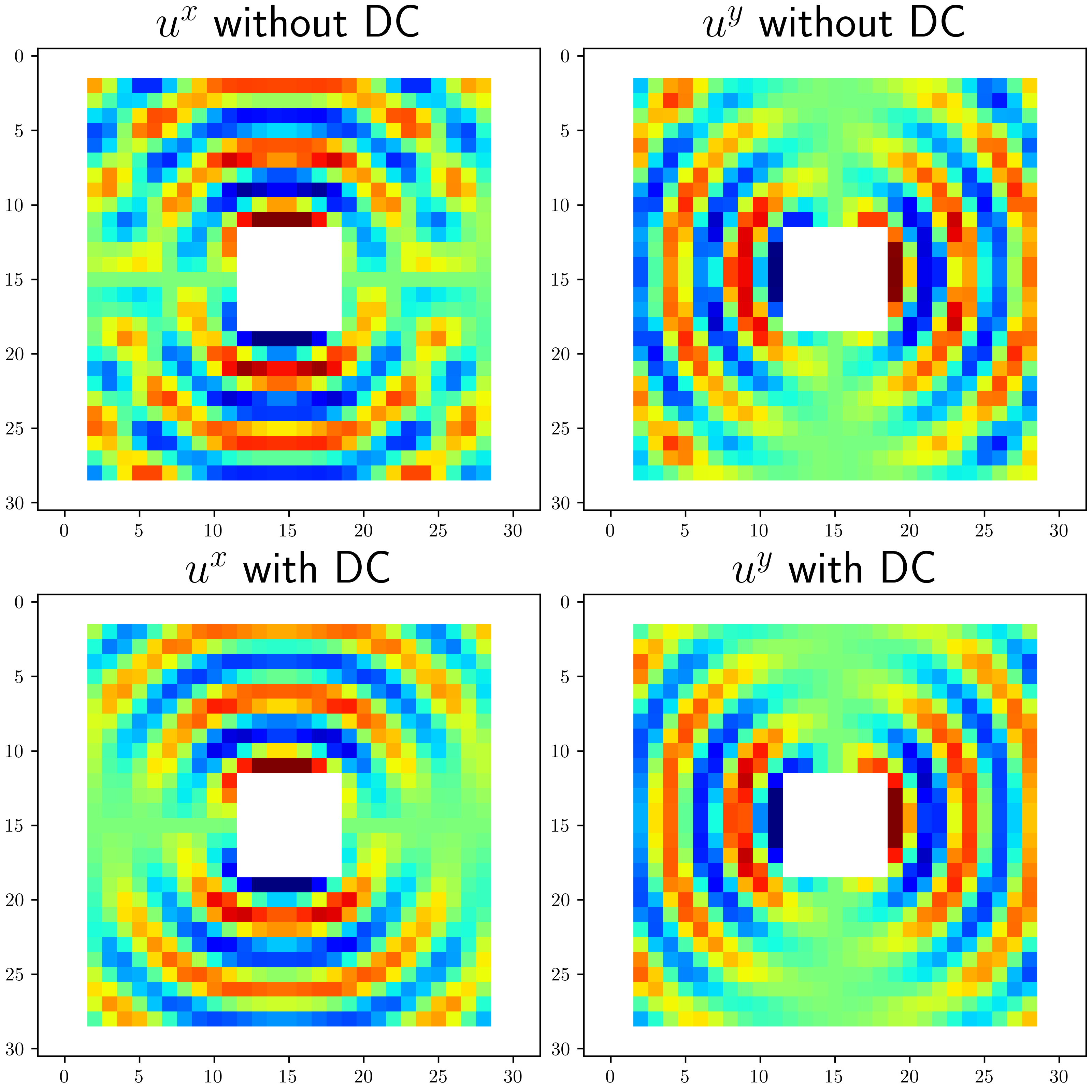}\label{fig:uxuy-n=30}} \quad
    \subfigure[\footnotesize Case $n=100$]{\image{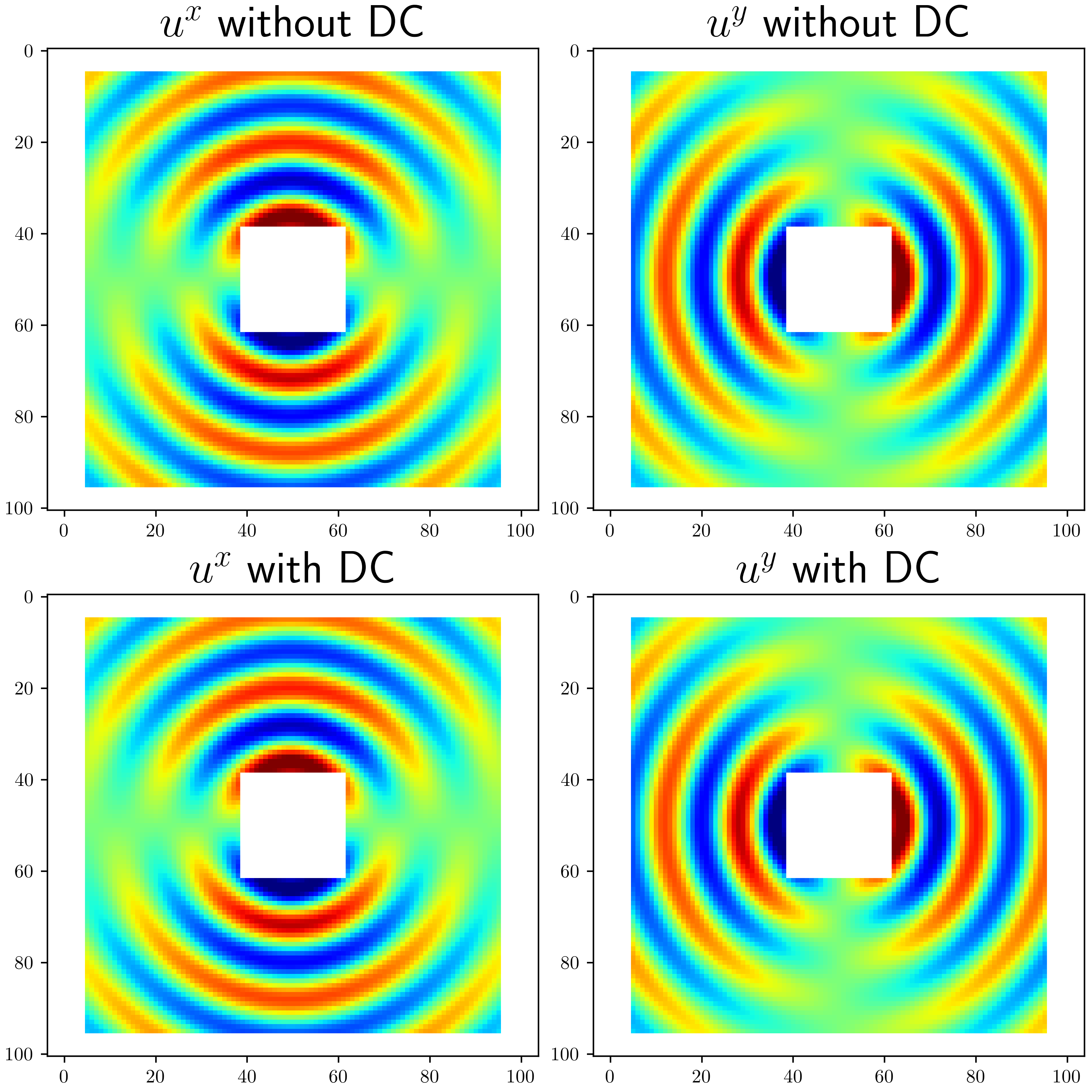}\label{fig:uxuy-n=100}}
\\
\end{center}
\caption{Example with analytical spherical S-waves. First line: without dispersion correction. Second line: with dispersion correction.
}\label{fig:illus-DC-S-Waves}
\end{figure}
\begin{figure}[h]
\begin{center}
    \includegraphics[width=4cm]{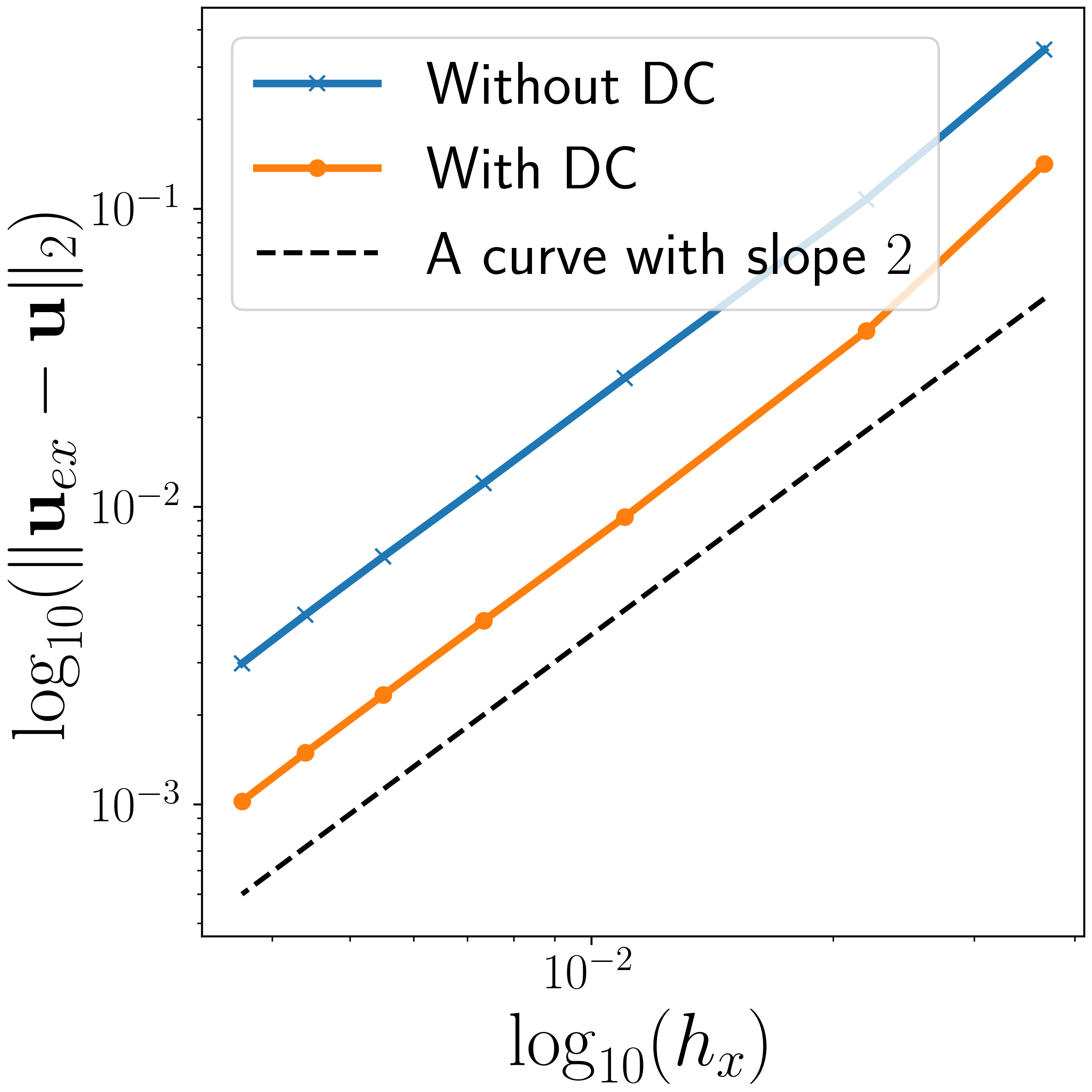}\qquad
    \includegraphics[width=4cm]{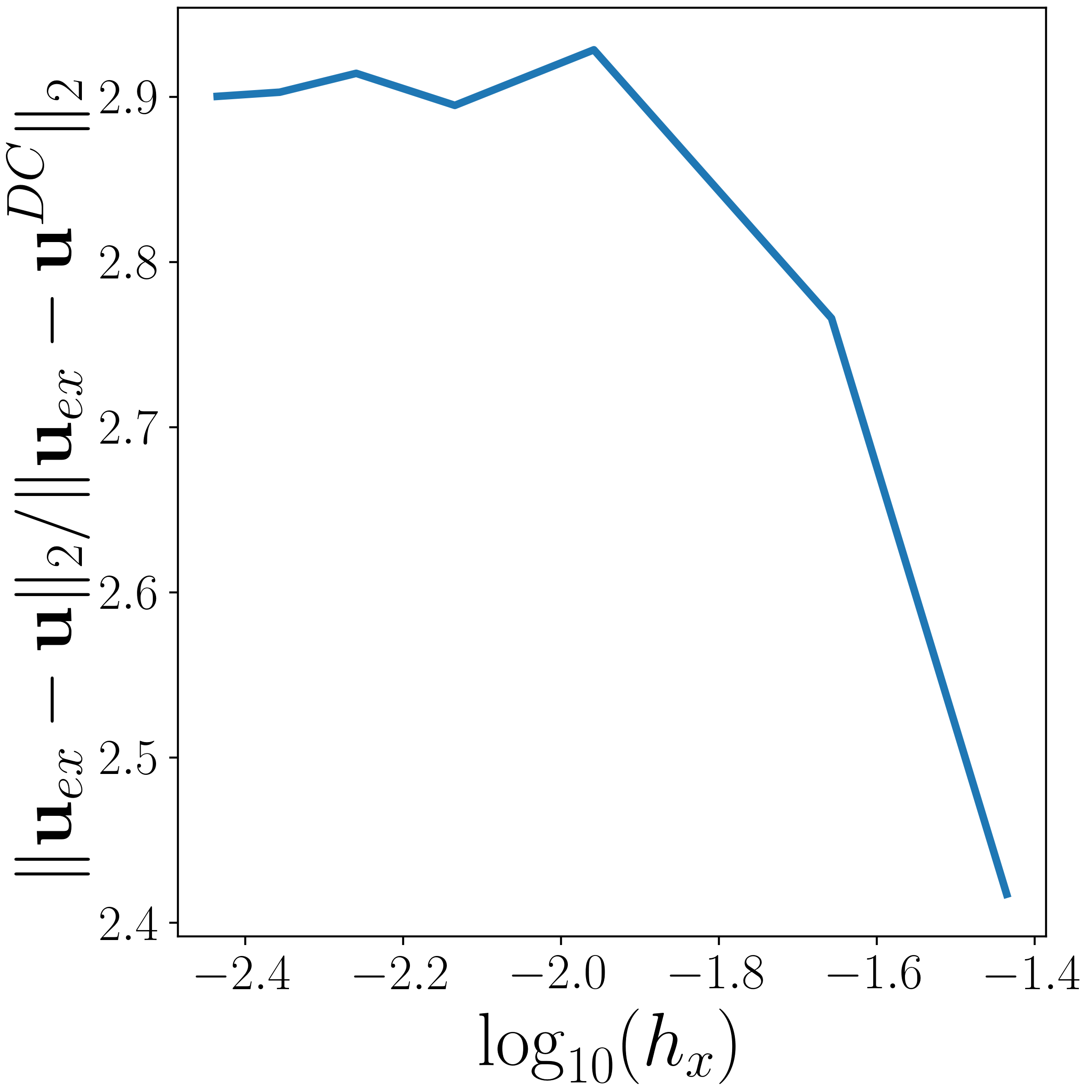}
    \caption{On the left: convergence curve without dispersion correction (blue curve) and with dispersion correction (orange curve). On the right: ratio between the error with and without dispersion correction where values $>1$ mean that the error with dispersion correction is lower that the one without.
    }
\end{center}
\end{figure}\label{fig:illus-DC-S-Waves-conv}

To assess the impact of the frequency, we have also tested for various $\omega$ the impact of DC on the same configuration. The results are reported in Table \ref{tab:Green_ww}. As we can see, for low frequency, DC is useless but it is very interesting for moderate or large frequencies, which are the cases of interest. It allows to divide the error by $3$ (at no additional cost, just by using the modified coefficients $\hat{\lambda}$ and $\hat{\mu}$).

\begin{table}
\centering
\begin{tabular}{|l|l|l|l|l|l|}
  \hline
  \multicolumn{6}{|c|}{$ \mathbf{\omega = 20}$} \\
  \hline
  $G_S$ & $4.78$ & 	$9.91$ & 	$14.70$ & 	$19.82$ & 	$29.73$ \\ 
  \hline
  Err. No DC & $0.2386$ &	$0.0880$ &	$0.0179$ &	$0.0110$ &	$0.0061$\\ 
  Err. With DC & $0.0678$ &	$0.0771$ &	$0.0173$ &	$0.0104$ &	$0.0057$ \\ 
  \hline
  Ratio & $\mathbf{3.5179}$ &	$1.1417$ &	$1.0349$ &	$1.0545$ &	$1.0685$ \\ 
  \hline
  \hline 
  \multicolumn{6}{|c|}{$ \mathbf{\omega = 50}$} \\
  \hline
  $G_S$ & $4.85$ & 	$9.86$ & 	$14.86$ & 	$19.87$ & 	$29.89$ \\ 
  \hline
  Err. No DC & $0.3420$ &	$0.0737$ &	$0.0325$ &	$0.0179$ &	$0.0078$\\ 
  Err. With DC & $0.1415$ &	$0.0253$ &	$0.0113$ &	$0.0061$ &	$0.0027$ \\ 
  \hline
  Ratio & $\mathbf{2.4172}$ &	$\mathbf{2.914}$ &	$\mathbf{2.8661}$ &	$\mathbf{2.924}$ &	$\mathbf{2.9052}$ \\ 
  \hline
  \hline
  \multicolumn{6}{|c|}{$ \mathbf{\omega = 100}$} \\
  \hline
  $G_S$ & $4.93$ & 	$9.94$ & 	$14.94$ & 	$19.95$ & 	$29.97$  \\ 
  \hline
  Err. No DC & $0.5296$ &	$0.1090$ &	$0.0459$ &	$0.0255$ &	$0.0114$ \\ 
  Err. With DC & $0.1518$ &	$0.0302$ &	$0.0135$ &	$0.0074$ &	$0.0032$ \\ 
  \hline
  Ratio & $\mathbf{3.4885}$ &	$\mathbf{3.6101}$ &	$\mathbf{3.4079}$ &	$\mathbf{3.4385}$ &	$\mathbf{3.5331}$ \\ 
  \hline
\end{tabular}
\caption{Error for various $\omega \in \{10, 50, 80, 100\}$ with or without dispersion correction (the error is truncated after $10^{-5}$). The ratio between the error with and without dispersion correction is calculated as before (see Table \ref{tab:OOA}, so that values $>1$ mean that the error with dispersion correction is lower that the one without (and are marked in bold).}
\label{tab:Green_ww}
\end{table}

\subsection{Effect on a multigrid algorithm}

We investigate here the effect of the dispersion correction on the behavior of the multigrid (MG) algorithm. 
This is motivated by some previous works which deal with the Helmholtz equation (see e.g. 
\cite{stolk2014multigrid,stolk2025two,cocquet2021closed,ernst2013multigrid,cocquet2018dispersion}) and show that reducing the dispersion error enhance the convergence of the MG method.

First we describe the multigrid setup. 
We apply geometric multigrid with re-discretized coarse operators to the discrete elastic Helmholtz equation in mixed formulation
\begin{equation}\label{eq:Elastic_Helm_mixed}
  \begin{pmatrix}
        -\mathbf{div}\left(\mu\nabla\right) -\rho\omega^2 & \nabla \\
        -\mathbf{div} & -\frac{1}{\lambda+\mu}
    \end{pmatrix}
    \begin{pmatrix}
        \vec u \\
        p
    \end{pmatrix}
    =
    \begin{pmatrix}
    \vec{f}\\
        0
    \end{pmatrix}
\end{equation}
which is equivalent to \eqref{eq:elastic_Helm_gradiv_formulation}, see \cite{treister2024hybrid}.
It was shown there that the mixed formulation enables faster solution using shifted Laplacian multigrid preconditioner.
It was also shown there that using the mixed formulation together with Vanka cell-wise smoothing, the method is scalable w.r.t. the Poisson ratio, nevertheless, a scallabiliy study that separates the formulation and the smoother was not performed. 
We use the intergrid operators suggested in \cite{treister2024hybrid}, that is, bilinear interpolation and full-weighting restriction (see \cite{trottenberg2000multigrid}) for the nodal direction, and biquadratic interpolation and nearest-neighbor restriction for the cell-centered direction. 
To demonstrate the efficiency of the dispersion correction, we use re-discretized coarse operators, where the Lam{\'e} coefficients are shifted according to Theorem \ref{thm:Asymptotic_shift_2nd_order}.
It was shown in \cite{yovel2024lfa} that discretizations with wider stencils can reduce the dispersion and improve multigrid convergence in a re-discretized monolithic multigrid method for the elastic Helmholtz equation.
Here, we show that even with a 5-point stencil significant improvement in multigrid convergence can be achieved.
Nevertheless, it is worth noting that block-preconditioned multigrid might be more efficient than monolithic multigrid for this system of equations, see \cite{yovel2024block}, and that Galerkin coarse approximation is generally more stable than re-discretization. 
We leave the application of dispersion correction to the Galerkin operator as a topic for future research.

In Figure \ref{fig:mg_conv}, we solve the elastic Helmholtz equation in mixed formulation \eqref{eq:Elastic_Helm_mixed}, discretized with the MAC scheme\footnote{The MAC scheme described in Section \ref{sec:mac} can be similarly applied to the mixed formulation, see detailed description in \cite{yovel2024block}.} in the dimensionless domain $[0,1]\times[0,1]$ with Neumann boundary conditions equiped with an absorbing boundary layer of increasing attenuation of 20 cells.
We solve the problem using multigrid two-level cycles with 2 pre- and 2 post-relaxations using damped Jacobi or 1 pre- and 1 post-relaxation of cell-wise Vanka (as shown in \cite{yovel2024lfa}, Vanka smoothing has the cost of about two Jacobi relaxations, so the methods are comparable).
We use Lam{\'e} coefficients are $\lambda=\mu=1$ and the density is $\rho=1$, yielding Poisson ratio $0.25$, and the same $\mu$ and $\rho$ with $\lambda=4$ yielding Poisson ratio $0.4$.
For the case of linear media, the Lam{\'e} coefficients and the density vary linearly in the ranges $\rho\in[2,3],$ $\lambda\in[4,20]$ and $\mu\in[1,15]$, yielding a Poisson ratio that ranges in $[0.105,0.476].$
The damping parameters were chosen by trial and error to be 0.4 for Jacobi and $0.6$ for cell-wise Vanka. 
Since we use multigrid with no complex shift (that is, no added artificial attenuation), the convergence is relatively slow.
We note that our aim in this paper is \emph{not} to design the perfect multigrid cycle, but to compare the performance of a cycle with standard components with and without dispersion correction.
To test the dependence of the convergence on the number of grid points per shear wavelength $G_s$ on the fine grid, we fix a constant angular frequency of $\omega = 21.33\pi$ for the homogeneous media and $\omega=15.085\pi$ for the linear media, and vary $G_s$ through the grid size.
Figure \ref{fig:mg_conv_const} shows that the dispersion correction improves the convergence significantly, and enables convergence even in $G_s$ values that would otherwise lead to divergent method. 
Nevertheless, in this case the difference between Jacobi and Vanka is not significant.
Figure \ref{fig:mg_conv_const_high_Poisson} shows that the Jacobi smoother does not converge for the higher Poisson ratio case. 
This result is interesting since it implies that in the monolithic multigrid framework, Vanka smoothing is necessary to achieve scalability with respect to the Poisson ratio.
Similar improvement by the dispersion correction is shown for linear media in Fig. \ref{fig:mg_conv_linear}, despite the fact that our assymptotical calculations in the proof of Theorem \ref{thm:Asymptotic_shift_2nd_order} assume constant coefficients.
Finally, Fig. \ref{fig:mg_conv_const_3D} shows that for the damped Jacobi smoother, the analogous three-dimensional dispersion correction from Theorem \ref{thm:Asymptotic_shift_2nd_order} behaves the same as in Fig. \ref{fig:mg_conv_const}.

\begin{figure}
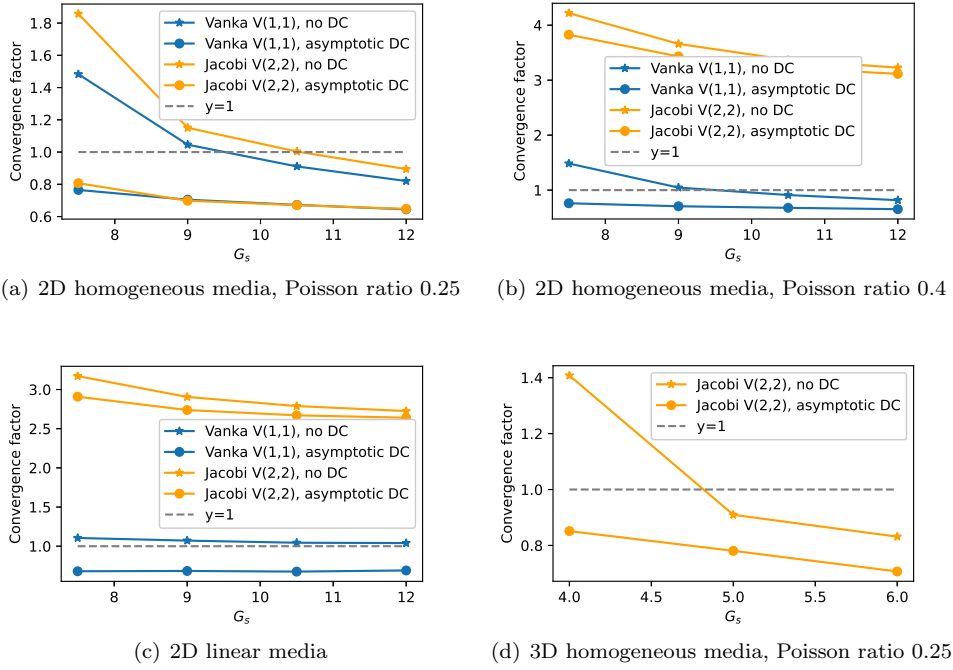

\begin{center}
	\newcommand{\image}[1]{\includegraphics[width=0.49\linewidth]{./#1}}
   \subfigure[\footnotesize 2D homogeneous media, Poisson ratio 0.25]{\image{mg_conv_const.eps}\label{fig:mg_conv_const}}
   \subfigure[\footnotesize 2D homogeneous media, Poisson ratio 0.4]{\image{mg_conv_const_high_Poisson.eps}\label{fig:mg_conv_const_high_Poisson}}  
   \subfigure[\footnotesize 2D linear media]{\image{mg_conv_linear.eps}\label{fig:mg_conv_linear}}
   \subfigure[\footnotesize 3D homogeneous media, Poisson ratio 0.25]{\image{mg_conv_const_3D.eps}\label{fig:mg_conv_const_3D}}\\
\end{center}
\caption{On the top left, two-level multigrid convergence with or without dispersion correction, for the Helmholtz equation in constant coefficients $\lambda = \mu = \rho = 1$ with frequency $\omega = 21.33\pi$ and varying grid size that gives different number of grid points per shear wavelength $G_s$ on the fine grid. 
On the top right, the same experiment is repeated with $\lambda=4$.
On the bottom left, a similar experiment is repeated for linear media with $\rho\in[2,3],$ $\lambda\in[4,20]$ and $\mu\in[1,15]$ for frequency $\omega = 15.08\pi,$ and on the bottom right, a similar three-dimensional experiment is repeated for homogeneous media with $\rho=\lambda=\mu=1$ and $\omega=8\pi.$
}\label{fig:mg_conv}
\end{figure}

    \subsection{A medium with heterogeneous density}

To finish this section on numerical results, let us consider a $2$D example with varying density $\rho = \rho(x,y)$. To apply DC, we simply use the obtained formula in the homogeneous case point-wise, that is to say, we take:
\[
    \widehat{\lambda}(x,y) = \lambda -h^2\dfrac{\rho(x,y) \omega^2}{16} \quad \text{and} \quad \widehat{\mu}(x,y) = \mu + h^2\dfrac{\rho(x,y) \omega^2}{16}.
\]
For our numerical test, we used a density inspired by the Marmousi model \footnote{
We note that \cite{martin2006marmousi2} suggests the Marmousi2 model, a physically valid elastic extension of the original Marmousi model \cite{brougois1990marmousi} with variable $\lambda$ and $\mu$.
We simply used the original Marmousi model only for varying $\rho$, since the fully variable elastic case, with jumping Lam{\'e} coefficients, is not feasible for our asymptotic analysis.
}. 
In Figure \ref{fig:illus-DC-S-Waves-Marmousi}, we illustrate the variable density $\rho$ and the results obtained with or without DC. These results have been obtained with $\omega = 70$, $\rho \in [1,6]$, $\lambda = 16$, $\mu = 2$ and $h := h_{\rm min} = \frac{1}{800}$. The computational domain is a rectangle $[0,2] \times [0,1]$. For the source term, we take a localized shear-waves centered in the middle of the computational domain. For the boundary conditions, we used simple ABC as above in subsection \ref{subsec:comp-analytical}.
\begin{figure}[h]
\begin{center}
    \includegraphics[width=12cm]{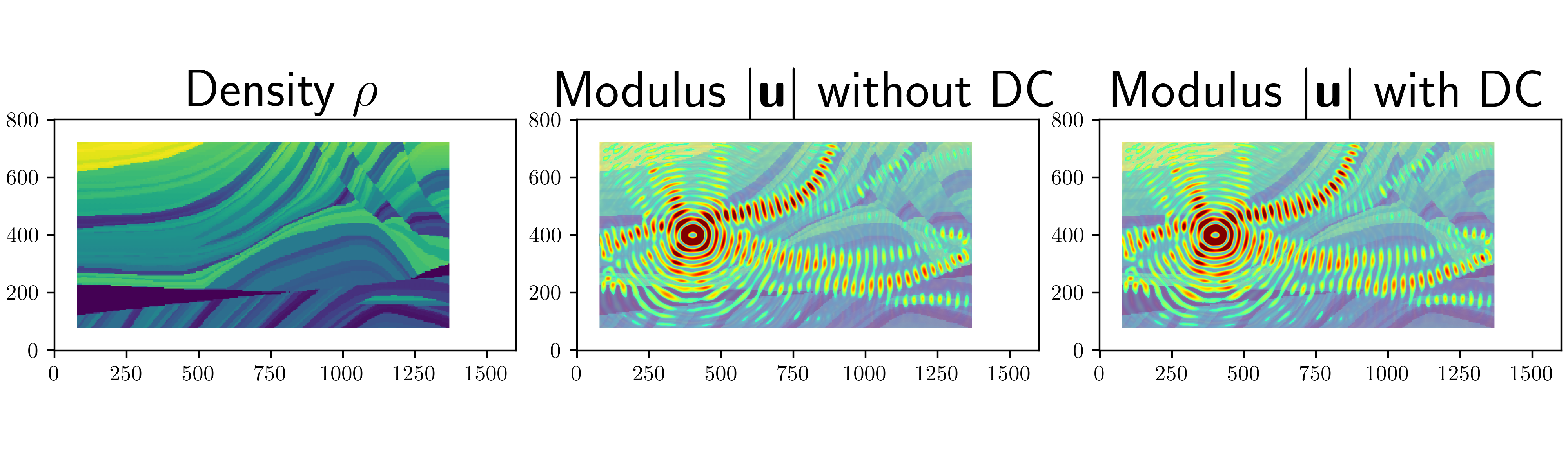}
    \caption{From left to right: density $\rho$, the solution obtained with DC and the solution obtained without DC.}
\end{center}
\end{figure}\label{fig:illus-DC-S-Waves-Marmousi}

Taking the case without DC as reference solution, we have computed the error as before (see subsection \ref{subsec:comp-analytical}) for various values of the discretization step $h$. The results are reported in Table \ref{tab:Marmousi_ww}. We can observe that, although the medium is heterogeneous, DC can notably improve the results for reasonable $h$. Let us emphasize that it is expected that ratio deteriorates as $h$ tends to $h_{\rm min}$ since we take as reference solution the solution without DC (and we therefore have a null error without DC for $h = h_{\rm min}$). Also, for small $h$ ($=32h_{\rm min}$ or $h=16 h_{\rm min}$) and $\omega = 70$, both solutions with or without DC are completly wrong so the error computation is not relevant (we have not a sufficient number of points per wavelength, especially in the region where $\rho$ is closed to $6$).  
\begin{table}
\centering
\begin{tabular}{|l|l|l|l|l|l|}
  \hline
  \multicolumn{6}{|c|}{$ \mathbf{\omega = 50}$} \\
  \hline
  $h/h_{\rm min}$ & $32$ & 	$16$ & 	$8$ & 	$4$ & 	$2$ \\
  \hline
  Err. No DC & $0.4028$ &	$0.0623$ &	$0.0243$ &	$0.0083$ &	$0.0022$\\ 
  Err. With DC & $0.137$ &	$0.0314$ &	$0.0136$ &	$0.0059$ &	$0.0022$ \\ 
  \hline
  Ratio & $\mathbf{2.9405}$ &	$\mathbf{1.9863}$ &	$\mathbf{1.7844}$ &	$\mathbf{1.4067}$ &	$0.9836$ \\ 
  \hline
  \hline 
  \multicolumn{6}{|c|}{$ \mathbf{\omega = 70}$} \\
  \hline
  $h/h_{\rm min}$ & $32$ & 	$16$ & 	$8$ & 	$4$ & 	$2$ \\ 
  \hline
  Err. No DC & $0.0409$ &	$0.0574$ &	$0.0330$ &	$0.0179$ &	$0.0035$\\ 
  Err. With DC & $0.1252$ &	$0.0527$ &	$0.013$ &	$0.0056$ &	$0.0018$ \\ 
  \hline
  Ratio & $0.3262$ &	$1.088$ &	$\mathbf{2.5262}$ &	$\mathbf{3.2154}$ &	$\mathbf{1.9157}$ \\ 
  \hline
\end{tabular}
\caption{Error for various $\omega \in \{50, 70\}$ with or without dispersion correction (the error is truncated after $10^{-5}$). The ratio between the error with and without dispersion correction is calculated as before (see Table \ref{tab:OOA}, so that values $>1$ mean that the error with dispersion correction is lower that the one without (and are marked in bold). We also recall that $h_{\rm min}$ corresponds to the discretization step for the reference solution computed without DC.}
\label{tab:Marmousi_ww}
\end{table}

\section{Conclusion and perspectives}

We have developed a new asymptotic dispersion correction for the MAC discretization of the isotropic elastic Helmholtz equation.
The correction is derived from an asymptotic analysis of the discrete dispersion relation and significantly reduces the leading-order phase error while preserving the structure of the original discretization.
As part of the analysis, we also introduced a decomposition and classification of dispersion relations for a broad class of elastic discretizations, providing a unified framework for their study.
Numerical experiments in two and three dimensions confirm the improved dispersion properties of the corrected MAC scheme and its favorable impact on multigrid convergence.

Several directions remain for future work.
A theoretical analysis establishing that the proposed dispersion correction indeed reduces the relative error would provide a rigorous complement to the present study.
Another natural extension is the development of dispersion corrections for anisotropic elasticity.
It would also be of interest to extend the approach to heterogeneous media by taking into account spatial variations of the material coefficients, including their gradients.
Finally, we leave the application of the proposed dispersion correction to the Galerkin operator as a topic for future research.

\bibliographystyle{siam}

\bibliography{AsymptoticDispersionElastic.bbl}

\begin{thebibliography}{10}

\bibitem{albella2018solving}
{\sc J.~Albella~Mart{\'\i}nez, S.~Imperiale, P.~Joly, and J.~Rodr{\'\i}guez},
  {\em Solving 2d linear isotropic elastodynamics by means of scalar
  potentials: a new challenge for finite elements}, Journal of Scientific
  Computing, 77 (2018), pp.~1832--1873.

\bibitem{babuska1997pollution}
{\sc I.~M. Babuska and S.~A. Sauter}, {\em Is the pollution effect of the fem
  avoidable for the helmholtz equation considering high wave numbers?}, SIAM
  Journal on numerical analysis, 34 (1997), pp.~2392--2423.

\bibitem{baronian2018linear}
{\sc V.~Baronian, L.~Bourgeois, B.~Chapuis, and A.~Recoquillay}, {\em Linear
  sampling method applied to non destructive testing of an elastic waveguide:
  theory, numerics and experiments}, Inverse Problems, 34 (2018), p.~075006.

\bibitem{berenger1994perfectly}
{\sc J.-P. Berenger}, {\em A perfectly matched layer for the absorption of
  electromagnetic waves}, Journal of Computational Physics, 114 (1994),
  pp.~185--200.

\bibitem{blitz1995ultrasonic}
{\sc J.~Blitz and G.~Simpson}, {\em Ultrasonic methods of non-destructive
  testing}, vol.~2, Springer Science \& Business Media, 1995.

\bibitem{bramble2008note}
{\sc J.~H. Bramble and J.~E. Pasciak}, {\em A note on the existence and
  uniqueness of solutions of frequency domain elastic wave problems: a priori
  estimates in h1}, Journal of mathematical analysis and applications, 345
  (2008), pp.~396--404.

\bibitem{brougois1990marmousi}
{\sc A.~Brougois, M.~Bourget, P.~Lailly, M.~Poulet, P.~Ricarte, and
  R.~Versteeg}, {\em Marmousi, model and data}, in EAEG workshop-practical
  aspects of seismic data inversion, European Association of Geoscientists \&
  Engineers, 1990, pp.~cp--108.

\bibitem{burel2012solving}
{\sc A.~Burel, S.~Imperiale, and P.~Joly}, {\em Solving the homogeneous
  isotropic linear elastodynamics equations using potentials and finite
  elements. the case of the rigid boundary condition}, Numerical Analysis and
  Applications, 5 (2012), pp.~136--143.

\bibitem{chen2012dispersion}
{\sc Z.~Chen, D.~Cheng, and T.~Wu}, {\em A dispersion minimizing finite
  difference scheme and preconditioned solver for the {3D} {Helmholtz}
  equation}, Journal of Computational Physics, 231 (2012), pp.~8152--8175.

\bibitem{cheng2017dispersion}
{\sc D.~Cheng, X.~Tan, and T.~Zeng}, {\em A dispersion minimizing finite
  difference scheme for the helmholtz equation based on point-weighting},
  Computers \& Mathematics with Applications, 73 (2017), pp.~2345--2359.

\bibitem{cocquet2024asymptotic}
{\sc P.-H. Cocquet and M.~J. Gander}, {\em Asymptotic dispersion correction in
  general finite difference schemes for helmholtz problems}, SIAM Journal on
  Scientific Computing, 46 (2024), pp.~A670--A696.

\bibitem{cocquet2025improving}
\leavevmode\vrule height 2pt depth -1.6pt width 23pt, {\em Improving yee’s
  scheme with asymptotic dispersion correction for time-harmonic maxwell’s
  equations}, Journal of Computational Physics,  (2025), p.~114602.

\bibitem{cocquet2018dispersion}
{\sc P.-H. Cocquet, M.~J. Gander, and X.~Xiang}, {\em Dispersion correction for
  helmholtz in 1d with piecewise constant wavenumber}, in International
  Conference on Domain Decomposition Methods, Springer, 2018, pp.~359--366.

\bibitem{cocquet2021closed}
\leavevmode\vrule height 2pt depth -1.6pt width 23pt, {\em Closed form
  dispersion corrections including a real shifted wavenumber for finite
  difference discretizations of 2d constant coefficient helmholtz problems},
  SIAM Journal on Scientific Computing, 43 (2021), pp.~A278--A308.

\bibitem{cocquet2012existence}
{\sc P.-H. Cocquet, P.-A. Mazet, and V.~Mouysset}, {\em On the existence and
  uniqueness of a solution for some frequency-dependent partial differential
  equations coming from the modeling of metamaterials}, SIAM Journal on
  Mathematical Analysis, 44 (2012), pp.~3806--3833.

\bibitem{engquist1977absorbing}
{\sc B.~Engquist and A.~Majda}, {\em Absorbing boundary conditions for
  numerical simulation of waves}, Proceedings of the National Academy of
  Sciences, 74 (1977), pp.~1765--1766.

\bibitem{ernst2013multigrid}
{\sc O.~G. Ernst and M.~J. Gander}, {\em Multigrid methods for helmholtz
  problems: A convergent scheme in 1d using standard components}, Direct and
  Inverse Problems in Wave Propagation and Applications. De Gruyer,  (2013),
  pp.~135--186.

\bibitem{galkowski2023does}
{\sc J.~Galkowski and E.~A. Spence}, {\em Does the helmholtz boundary element
  method suffer from the pollution effect?}, Siam Review, 65 (2023),
  pp.~806--828.

\bibitem{gander2026fourier}
{\sc M.~J. Gander, H.~Zhang, and H.~Zhou}, {\em Fourier analysis of finite
  difference schemes for the helmholtz equation in 1d with dirichlet
  conditions: Sharp estimates and relative errors}, Numerical Algorithms,
  (2026), pp.~1--54.

\bibitem{gercek2007poisson}
{\sc H.~Gercek}, {\em Poisson's ratio values for rocks}, International Journal
  of Rock Mechanics and Mining Sciences, 44 (2007), pp.~1--13.

\bibitem{gosselin20143d}
{\sc B.~Gosselin-Cliche and B.~Giroux}, {\em {3D} frequency-domain
  finite-difference viscoelastic-wave modeling using weighted average 27-point
  operators with optimal coefficients}, Geophysics, 79 (2014), pp.~T169--T188.

\bibitem{harari2011stabilized}
{\sc I.~Harari, R.~Ganel, and E.~Grosu}, {\em Stabilized finite elements for
  time-harmonic elastic waves}, Computer methods in applied mechanics and
  engineering, 200 (2011), pp.~1774--1786.

\bibitem{ihlenburg1995dispersion}
{\sc F.~Ihlenburg and I.~Babu{\v{s}}ka}, {\em Dispersion analysis and error
  estimation of galerkin finite element methods for the helmholtz equation},
  International journal for numerical methods in engineering, 38 (1995),
  pp.~3745--3774.

\bibitem{jo1996optimal}
{\sc C.-H. Jo, C.~Shin, and J.~H. Suh}, {\em An optimal 9-point,
  finite-difference, frequency-space, {2-D} scalar wave extrapolator},
  Geophysics, 61 (1996), pp.~529--537.

\bibitem{levander1988fourth}
{\sc A.~R. Levander}, {\em Fourth-order finite-difference {P-SV} seismograms},
  Geophysics, 53 (1988), pp.~1425--1436.

\bibitem{liu2015optimized}
{\sc Z.-l. Liu, P.~Song, J.-s. Li, J.~Li, and X.-b. Zhang}, {\em An optimized
  implicit finite-difference scheme for the two-dimensional helmholtz
  equation}, Geophysical Journal International, 202 (2015), pp.~1805--1826.

\bibitem{martin2006marmousi2}
{\sc G.~S. Martin, R.~Wiley, and K.~J. Marfurt}, {\em Marmousi2: An elastic
  upgrade for {Marmousi}}, The Leading Edge, 25 (2006), pp.~156--166.

\bibitem{martin2006multiple}
{\sc P.~A. Martin}, {\em Multiple scattering: interaction of time-harmonic
  waves with N obstacles}, no.~107, Cambridge University Press, 2006.

\bibitem{mattesi2019high}
{\sc V.~Mattesi, M.~Darbas, and C.~Geuzaine}, {\em A high-order absorbing
  boundary condition for 2d time-harmonic elastodynamic scattering problems},
  Computers \& Mathematics with Applications, 77 (2019), pp.~1703--1721.

\bibitem{nilsson2007stable}
{\sc S.~Nilsson, N.~A. Petersson, B.~Sj{\"o}green, and H.-O. Kreiss}, {\em
  Stable difference approximations for the elastic wave equation in second
  order formulation}, SIAM Journal on Numerical Analysis, 45 (2007),
  pp.~1902--1936.

\bibitem{operto2009finite}
{\sc S.~Operto, J.~Virieux, A.~Ribodetti, and J.~E. Anderson}, {\em
  Finite-difference frequency-domain modeling of viscoacoustic wave propagation
  in 2d tilted transversely isotropic (tti) media}, Geophysics, 74 (2009),
  pp.~T75--T95.

\bibitem{rui2018locking}
{\sc H.~Rui and M.~Sun}, {\em A locking-free finite difference method on
  staggered grids for linear elasticity problems}, Computers \& Mathematics
  with Applications, 76 (2018), pp.~1301--1320.

\bibitem{sjogreen2012fourth}
{\sc B.~Sj{\"o}green and N.~A. Petersson}, {\em A fourth order accurate finite
  difference scheme for the elastic wave equation in second order formulation},
  Journal of Scientific Computing, 52 (2012), pp.~17--48.

\bibitem{spence2023simple}
{\sc E.~A. Spence}, {\em A simple proof that the hp-fem does not suffer from
  the pollution effect for the constant-coefficient full-space helmholtz
  equation}, Advances in Computational Mathematics, 49 (2023), p.~27.

\bibitem{stolk2025two}
{\sc C.~C. Stolk}, {\em A two-grid method with dispersion matching for
  finite-element helmholtz problems: C. stolk}, Advances in Computational
  Mathematics, 51 (2025), p.~43.

\bibitem{stolk2014multigrid}
{\sc C.~C. Stolk, M.~Ahmed, and S.~K. Bhowmik}, {\em A multigrid method for the
  {Helmholtz} equation with optimized coarse grid corrections}, SIAM Journal on
  Scientific Computing, 36 (2014), pp.~A2819--A2841.

\bibitem{treister2024hybrid}
{\sc E.~Treister and R.~Yovel}, {\em A hybrid shifted {Laplacian} multigrid and
  domain decomposition preconditioner for the elastic {Helmholtz} equations},
  Journal of Computational Physics, 497 (2024), p.~112622.

\bibitem{trottenberg2000multigrid}
{\sc U.~Trottenberg, C.~Oosterlee, and A.~Sch\"{u}ller}, {\em Multigrid},
  Academic Press, London and San Diego, 2001.

\bibitem{virieux2011review}
{\sc J.~Virieux, H.~Calandra, and R.-{\'E}. Plessix}, {\em A review of the
  spectral, pseudo-spectral, finite-difference and finite-element modelling
  techniques for geophysical imaging}, Geophysical Prospecting, 59 (2011),
  pp.~794--813.

\bibitem{wu2014dispersion}
{\sc T.~Wu and Z.~Chen}, {\em A dispersion minimizing subgridding finite
  difference scheme for the helmholtz equation with pml}, Journal of
  Computational and Applied Mathematics, 267 (2014), pp.~82--95.

\bibitem{wu2021new}
{\sc T.~Wu, Y.~Sun, and D.~Cheng}, {\em A new finite difference scheme for the
  3d helmholtz equation with a preconditioned iterative solver}, Applied
  Numerical Mathematics, 161 (2021), pp.~348--371.

\bibitem{wu2018optimal}
{\sc T.~Wu and R.~Xu}, {\em An optimal compact sixth-order finite difference
  scheme for the helmholtz equation}, Computers \& Mathematics with
  Applications, 75 (2018), pp.~2520--2537.

\bibitem{yovel2024block}
{\sc R.~Yovel and E.~Treister}, {\em A block-acoustic preconditioner for the
  elastic helmholtz equation}, arXiv preprint arXiv:2411.15897,  (2024).

\bibitem{yovel2024lfa}
\leavevmode\vrule height 2pt depth -1.6pt width 23pt, {\em {LFA}-tuned
  matrix-free multigrid method for the elastic {Helmholtz} equation}, SIAM
  Journal on Scientific Computing,  (2024), pp.~S1--S21.

\end{thebibliography}

\end{document}